%% file: main.tex
\documentclass[11pt]{article}
\input{preamble}

\title{A Yoneda-style embedding for virtual equipments}
\date{}
\author{David Jaz Myers}

\graphicspath{ {graphics/} }

\begin{document}

\input{Paper/Embed_Paper/Intro.tex}

\tableofcontents

\input{Paper/Embed_Paper/Embed}

\printbibliography

\end{document}

%% file: preamble.tex
\usepackage{fullpage}
\usepackage{epsfig}
\usepackage{amsmath,amssymb,amsthm}
\usepackage{mathtools}
\usepackage{tikz}
\usepackage{wrapfig}
\usepackage{graphicx}
\usepackage{tikz-cd}
\usepackage[export]{adjustbox}
\usepackage[cmtip,all]{xy}

\makeatletter
\DeclareRobustCommand{\cev}[1]{%
  \mathpalette\do@cev{#1}%
}
\newcommand{\do@cev}[2]{%
  \fix@cev{#1}{+}%
  \reflectbox{$\m@th#1\vec{\reflectbox{$\fix@cev{#1}{-}\m@th#1#2\fix@cev{#1}{+}$}}$}%
  \fix@cev{#1}{-}%
}
\newcommand{\fix@cev}[2]{%
  \ifx#1\displaystyle
    \mkern#23mu
  \else
    \ifx#1\textstyle
      \mkern#23mu
    \else
      \ifx#1\scriptstyle
        \mkern#22mu
      \else
        \mkern#22mu
      \fi
    \fi
  \fi
}

\makeatother

\makeatletter
\def\slashedarrowfill@#1#2#3#4#5{%
  $\m@th\thickmuskip0mu\medmuskip\thickmuskip\thinmuskip\thickmuskip
\relax#5#1\mkern-7mu%
\cleaders\hbox{$#5\mkern-2mu#2\mkern-2mu$}\hfill
\mathclap{#3}\mathclap{#2}%
\cleaders\hbox{$#5\mkern-2mu#2\mkern-2mu$}\hfill
\mkern-7mu#4$%
}
\def\rightslashedarrowfill@{%
  \slashedarrowfill@\relbar\relbar\mapstochar\rightarrow}
\newcommand\xslashedrightarrow[2][]{%
  \ext@arrow 0055{\rightslashedarrowfill@}{#1}{#2}}
\makeatother

\usepackage[backend=biber, style=alphabetic]{biblatex}
\newtheorem{definition}{Definition}

\newtheorem{remark}{Remark}

\newtheorem{lemma}{Lemma}
\newtheorem{construction}{Construction}

\newtheorem{theorem}{Theorem}
\newtheorem{corollary}{Corollary}
\newtheorem{proposition}{Proposition}

\let\P\relax
\let\v\relax
\DeclareMathOperator{\Ca}{\mathcal{C}}

\DeclareMathOperator{\Da}{\mathcal{D}}
\DeclareMathOperator{\Ea}{\mathcal{E}}

\DeclareMathOperator{\F}{\mathbb{F}}
\DeclareMathOperator{\Z}{\mathbb{Z}}
\DeclareMathOperator{\N}{\mathbb{N}}
\DeclareMathOperator{\R}{\mathbb{R}}
\DeclareMathOperator{\C}{\mathbb{C}}
\DeclareMathOperator{\G}{\mathbb{G}}
\DeclareMathOperator{\M}{\mathbb{M}}

\DeclareMathOperator{\id}{\textbf{id}}

\DeclareMathOperator{\U}{\mathcal{U}}

\DeclareMathOperator{\h}{\text{h}}
\DeclareMathOperator{\v}{\text{v}}

\DeclareMathOperator{\topro}{\xslashedrightarrow{}} %

\newcommand{\VCat}[1]{{#1}\textbf{-Cat}}
\newcommand{\nanograph}[1]{\begin{tabular}{@{}c@{}}\includegraphics[scale=.125]{#1}\end{tabular}}
\newcommand{\tinygraph}[1]{\begin{tabular}{@{}c@{}}\includegraphics[scale=.25]{#1}\end{tabular}}

\newcommand{\ingraph}[1]{\begin{tabular}{@{}c@{}}\includegraphics[scale=.5]{#1}\end{tabular}}
\newcommand{\fullgraph}[1]{\begin{tabular}{@{}c@{}}\includegraphics{#1}\end{tabular}}

\newcommand{\longsquiggly}{\xymatrix{{}\ar@{~>}[r]&{}}}

\newcommand{\comment}[1]{}

\newcommand{\realize}[1]{\left| {#1} \right|}

\newbox\gnBoxA
\newdimen\gnCornerHgt
\setbox\gnBoxA=\hbox{$\ulcorner$}
\global\gnCornerHgt=\ht\gnBoxA
\newdimen\gnArgHgt
\def\godelnum #1{%
\setbox\gnBoxA=\hbox{$#1$}%
\gnArgHgt=\ht\gnBoxA%
\ifnum     \gnArgHgt<\gnCornerHgt \gnArgHgt=0pt%
\else \advance \gnArgHgt by -\gnCornerHgt%
\fi \raise\gnArgHgt\hbox{$\ulcorner$} \box\gnBoxA %
\raise\gnArgHgt\hbox{$\urcorner$}}

%% file: Paper/Embed_Paper/Intro.tex
\maketitle

\begin{abstract}
    In this paper, we exhibit a ``Yoneda''-style embedding of any virtual equipment into the virtual equipment of categories enriched in it. We show that this embedding preserves composition, is full on 2-cells and arrows, and coreflective on proarrows.
\end{abstract}

\section{Introduction}

In his 1973 paper \cite{LawvereMetric}, Lawvere remarks
\begin{quote}
It is a banality that all mathematical structures of a given kind constitute the objects of a category; the sequence: elements/structures/categories thus has led some people to attempt to characterize the philosophical significance of the theory of categories as that of a ``third level of abstraction''. But the theory of categories actually penetrates much more deeply than that attempted characterization would suggest toward summing up the essence of mathematics. The kinds of structures which actually arise in the practice of geometry and analysis are far from being ``arbitrary'', and indeed in this paper we will investigate a particular case of the way in which logic should be specialized to take account of this experience of non-arbitrariness, as concentrated in the thesis that \emph{fundamental} structures are themselves categories. (\cite[p.~135]{LawvereMetric})
\end{quote}

In other words, not only are the  most fundamental structures of mathematics organized in categories, they are in many cases (enriched) categories themselves. In this paper, we will embed any virtual equipment into the virtual equipment of categories enriched in it. This gives a formal flair to Lawvere's thesis: so long as our objects of interest can be organized profitably into a virtual equipment, then they and their morphisms may be realized as enriched categories and functors.

Virtual equipments are a special sort of virtual double categories, which are to double categories as multicategories are to monoidal categories. Virtual double categories were introduced as \emph{multicat{\'e}gories} by Burroni \cite{Burroni}, and used as a general setting for enrichment as \emph{fc-multicategories} by Leinster \cite{Leinster}. In \cite{Cruttwell}, Cruttwell and Shulman coin the name ``virtual double category'' and define virtual equipments as certain sorts of virtual double categories especially suited to the study of generalized multicategories. 

Virtual equipments are a good setting in which to enrich categories, generalizing enrichment in monoidal categories, multicategories, bicategories, and equipments simultaneously. Furthermore, enriched categories with their functors and profunctors form virtual equipments --- even in cases where the enrichment base is not cocomplete enough to ensure the existence of composites of profunctors. We prove a converse of sorts to this statement, showing that if a class of structures forms a virtual equipment, then they can be realized as enriched categories of a sort. This gives an embedding of any virtual equipment into the virutal equipment of categories enriched in it. 

\begin{construction}
There is a ``Yoneda''-style embedding $| \cdot | : \Ea \rightarrow \VCat{\Ea}$ of a virtual equipment $\Ea$ into the virtual equipment of categories enriched in it.
\end{construction}

  This construction specializes nicely to familiar cases. If the virtual equipment is a monoidal equipment, as the equipments of categories enriched in a suitably cocomplete monoidal category are, then we can restrict our construction to get a double functor $\Ea \rightarrow \VCat{h\Ea(1, 1)}$ from $\Ea$ to the virtual equipment of categories enriched in the monoidal category $h\Ea(1, 1)$ of proarrows on the monoidal identity $1$. For example, if $\Ea = \textbf{Ring}$, the equipment of rings, homomorphisms, and bimodules, then $|\cdot|$ so restricted interprets each ring as a single object category enriched in abelian groups. It will be clear from the construction that this situation will hold in general; if $\Ea = \VCat{\Ca}$ is the equipment of categories enriched in $\Ca$, then $h\Ea(1,1) \simeq \Ca$ and the restricted double functor $|\cdot| : \Ea \to \VCat{h\Ea(1,1)}$ will be an equivalence.

To describe the construction and prove our main theorem about it, we will use a graphical calculus for virtual equipments based on a similar calculus for equipments developed by the author \cite{myers2016string}. The calculus is an extension of the usual string diagrams for bicategories, and we will explain it in Section \ref{sec:virtual.equipments}.

In Section \ref{sec:enriched.cats}, we will use the graphical calculus to construct the virtual equipment of categories enriched in a virtual equipment. Then, in Section \ref{sec:embedding} we will construct the embedding. Finally, in Sections \ref{sec:functoriality} and \ref{sec:properties} we will prove the following properties of the embedding in our main theorem.

\begin{theorem}
The embedding $|\cdot| : \Ea \to \VCat{\Ea}$:
\begin{enumerate}
    \item Preserves composites,
    \item Is fully faithful on 2-cells,
    \item Is full on arrows, and
    \item Is coreflective on proarrows in the sense that the induced map $h\Ea(A, B) \to h(\VCat{\Ea})(|A|, |B|)$ is fully faithful and admits a right adjoint.
\end{enumerate}
\end{theorem}

As a corollary, $|\cdot|$ reflects both equivalence (in the $2$-category of arrows $v\Ea$) and Morita equivalence (in the bicategory of composable proarrows).

%% file: Paper/Embed_Paper/Embed.tex
\section{Virtual Equipments}\label{sec:virtual.equipments}
A virtual equipment is to an equipment what a multicategory is to a monoidal category; while in an equipment we may compose proarrows, in a virtual equipment we may only take formal composites of proarrows in the domains of our 2-cells. We move to virtual equipments to construct the canonical embedding because the existence of composites of enriched profunctors requires certain cocompleteness assumptions which do not hold in general equipments. Categories enriched in any virtual equipment will, however, form virtual equipment.

Virtual equipments are useful even in more traditional enrichment settings when the base is not suitably cocomplete. This may occur because the base of enrichment lacks colimits, but it can also occur if it admits all small colimits. If we allow profunctors over large categories, then they may fail to compose even if the enriching base is (small) cocomplete. Therefore, large\footnote{In the virtual double category of large categories and profunctors, a category admits a \emph{unit} (a nullary composite) if and only if it is locally small. So, in order to get a virtual equipment, we must restrict ourselves to locally small categories.} categories and profunctors form a virtual equipment, but not an equipment.

We begin with the notion of a virtual double category, which was introduced under the name \emph{multicat{\'e}gory} by Burroni in \cite{Burroni}, and used by Leinster as a general setting for enrichment under the name \emph{fc-multicategory} \cite{Leinster}. We will follow along with the presentation in sections 2 and 7 of Cruttwell and Shulman \cite{Cruttwell}.

A virtual double category is to a double category what a multicategory is to a monoidal category. That is, we are no longer allowed to compose proarrows in general, but we modify our 2-cells so that they may have formal composites of proarrows as their horizontal domain. We define a virtual double category now, using Definition 2.1 of Cruttwell and Shulman \cite{Cruttwell}.

\begin{definition}
A \emph{virtual double category} $\Da$ consists of:
\begin{enumerate}
    \item A category $v\Da$ of ``vertical arrows''. We draw the objects of this category as colored plane regions $\tinygraph{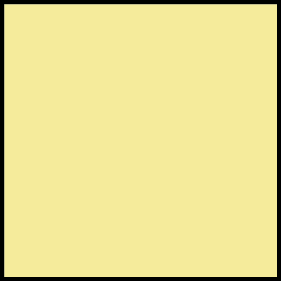}$, $\tinygraph{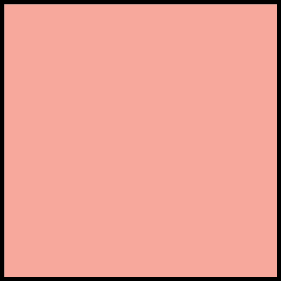}$, and draw arrows as vertical wires 
    $$\ingraph{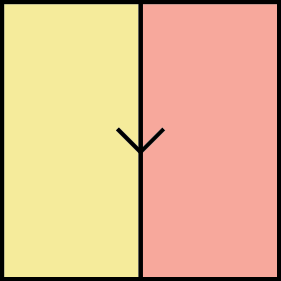} : \tinygraph{oY} \to \tinygraph{oR}$$ 
    separating the plane regions.
    \item For each two objects $\tinygraph{oY}$ and $\tinygraph{oR}$ of $v\Da$, a set $h\Da\left(\tinygraph{oY}, \tinygraph{oR}\right)$ of ``horizontal arrows'' or ``proarrows''. We draw a horizontal arrow as a horizontal wire 
    $$\ingraph{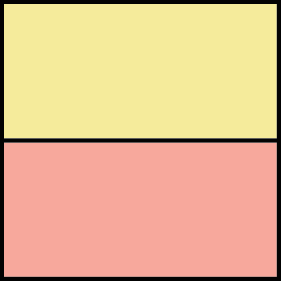} \in h\Da\left(\tinygraph{oY}, \tinygraph{oR}\right)$$
    separating the plane regions.
    \item A set of 2-cells $\alpha$, each of which as a vertical domain and codomain, a horizontal codomain, and a list of horizontal domains arranged as follows:
    \[\begin{tikzcd} X_0 \arrow{dd}[swap]{f} \arrow{r}{J_1} & X_1 \arrow{r}{J_2} & X_3 \arrow{r}{J_3} & \cdots \arrow{r}{J_k} & X_k \arrow{dd}{g}\\
    & & \alpha & & \\
    Y_0 \arrow{rrrr}[swap]{K} & & & & Y_k
    \end{tikzcd}\]
    We draw this as a node connecting the wires:
    $$\fullgraph{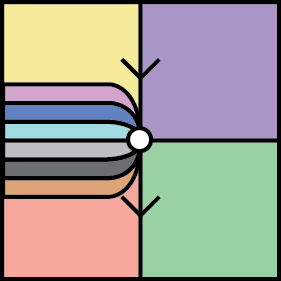}$$
    
    \item For each proarrow $J \in h\Da(X, Y)$, a 2-cell 
    \[\begin{tikzcd}
    X \arrow[equals]{d} \arrow{r}{J} & Y \arrow[equals]{d}\\
    X  \arrow{r}[swap]{J} & Y
    \end{tikzcd}\]
    We do not draw this 2-cell; it is represented in the same way as the proarrow $J$.
    
    \item For each arrangement
    \[\begin{tikzcd}
    X_{0} \arrow[dashed]{rr}{J_{11}\cdots J_{1k_1}} \arrow{dd} & & X_1 \arrow[dashed]{rr}{J_{21}\cdots J_{2k_2}} \arrow{dd} & & \cdots \arrow[dashed]{rr}{\cdots} \arrow{dd} & & X_{\ell} \arrow{dd}\\
    & \alpha_1 & & \alpha_2 &  & \cdots & \\
    Y_0 \arrow{dd} \arrow{rr}{K_1} & & Y_1 \arrow{rr}{K_2} & & \cdots \arrow{rr}{\cdots}& & Y_{\ell} \arrow{dd} \\
    & & & \beta &  & & \\
    Z_0 \arrow{rrrrrr}{H} & & & & & & Z_{\ell} \\
    \end{tikzcd}\]
    a 2-cell
    \[\begin{tikzcd}
    X_{0} \arrow[dashed]{rr}{J_{11}\cdots J_{1k_1}} \arrow{d}  & & X_1 \arrow[dashed]{rr}{J_{21}\cdots J_{2k_2}}  & & \cdots \arrow[dashed]{rr}{\cdots}  & & X_{\ell} \arrow{d}\\
    Y_0 \arrow{d}  & &  & \beta(\alpha_1, \alpha_2, \cdots) & & & Y_{\ell} \arrow{d} \\
    Z_0 \arrow{rrrrrr}{H} & & & & & & Z_{\ell} \\
    \end{tikzcd}\]
    We draw this by connecting the wires incident to the nodes. For example, we compose the 2-cell $\ingraph{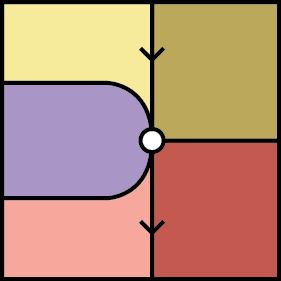}$ with $\ingraph{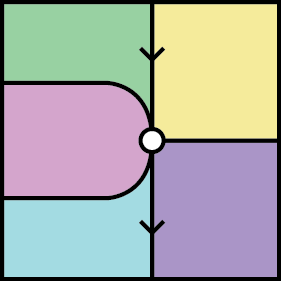}$ and $\ingraph{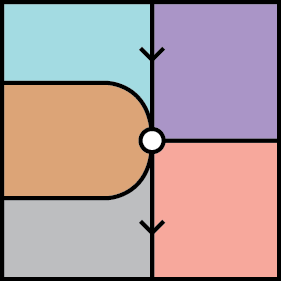}$ to get the composite
    $$\fullgraph{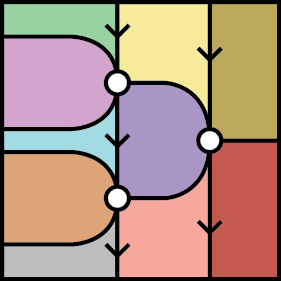}$$
    
    \item This data satisfies the identity and associativity axioms. These say, respectively:
    \begin{itemize}
        \item \textbf{Identity:} $\beta(\id_{J_1},\, \cdots,\, \id_{J_k}) = \beta$ and $\id_J(\alpha) = \alpha$.
        \item \textbf{Associativity:} $\beta\big(\alpha_1(\gamma_{11},\,\cdots,\,\gamma_{1q_1}),\,\cdots,\,\alpha_k(\gamma_{k1},\,\cdots,\,\gamma_{kq_k})\big) = \beta(\alpha_1,\,\cdots,\,\alpha_k)(\gamma_{11},\,\cdots,\,\gamma_{kq_k})$.
    \end{itemize}
    These laws guarantee that each diagram may be read as a unique 2-cell.
\end{enumerate}
\end{definition}

It should come as no surprise that every double category is also a virtual double category by taking the virtual 2-cells with a given list of proarrows as domain to be the actual 2-cells with the composite of that list as domain.

We can define the existence of an actual composite in a virtual double category with a universal property.
\begin{definition}
Given proarrows $J_1$, $\ldots$, $J_k$ and $J$ in a virtual double category, $J$ is the \emph{composite} of $J_1$, $\ldots$, $J_k$ if there is a 2-cell 
\[\begin{tikzcd} X_0 \arrow[equals]{d}\arrow{r}{J_1} & X_1 \arrow{r}{J_2} & X_3 \arrow{r}{J_3} & \cdots \arrow{r}{J_k} & X_k \arrow[equals]{d}\\
    X_0 \arrow{rrrr}[swap]{J} & & & & X_k
    \end{tikzcd}\]
such that every 2-cell
\[\begin{tikzcd} Z \arrow[dashed]{r}{K_i} \arrow{d} & X_0  \arrow{r}{J_1} & \cdots \arrow{r}{J_k} & X_k  \arrow[dashed]{r}{Q_i} & W \arrow{d}\\
    Y_0 \arrow{rrrr}[swap]{K} & & & & Y_k
    \end{tikzcd}\]
factors uniquely as
\[\begin{tikzcd}
    Z \arrow[dashed]{r}{K_i} \arrow[equals]{d} & X_0 \arrow[equals]{d} \arrow{r}{J_1} & \cdots \arrow{r}{J_k} & X_k \arrow[equals]{d} \arrow[dashed]{r}{Q_i} & W \arrow[equals]{d}\\
    Z \arrow[dashed]{r}{K_i} \arrow{d} & X_0 \arrow{rr}{J} & & X_k \arrow[dashed]{r}{Q_i} & W \arrow{d}\\
    Y_0 \arrow{rrrr}[swap]{K} & & & & Y_k
    \end{tikzcd}.\]
\end{definition}

The associativity of composites follows from the associativity of composition.
\begin{lemma}\label{Lem:VirtualComp}
Let $J_{ij}$ be sequence of proarrows in a virtual double category with composite $J$. Suppose that for fixed $i$, there is a composite $J_i$ of the $J_{ij}$. Then $J$ is the composite of the $J_i$.
\end{lemma}
\begin{proof}
By the universal property of $J_i$, the structure 2-cell $J_{kj} \to J$ factors through $J_{ij} \to J_i$. This gives a 2-cell from $J_i \to J$ for each $i$. Together, these witness $J$ as the composite of the $J_i$ by applying the universal property of $J$ as the composite of the $J_{ij}$ and then the universal property of the $J_i$.
\end{proof}

A virtual double category with all composites is equivalently a double category. However, there is a useful special case of composites which are very common in virtual double categories even when general composites do not exist: \emph{nullary} composites, or \emph{units}

\begin{definition}
For an object $A$, a \emph{unit} $h_A : A \topro A$ for $A$ is a composite of the empty list of proarrows starting and ending at $A$. 
\end{definition}
Because any 2-cell whose horizontal domain included $A$ factors uniquely through a 2-cell with $h_A$ inserted, we will simply not draw $h_A$ and refer to it in the same way as the object $A$. So, if $A$ is $\tinygraph{oY}$, then we will also refer to $h_A$ by $\tinygraph{oY}$. 

A virtual equipment is a virtual double category with all units and where proarrows can be \emph{restricted} along arrows. This allows proarrows to really function like \emph{bimodules} of categories, since they can be restricted along arrows --- that is, applied to generalized elements to yield new bimodules.
\begin{definition}
        A cell $\ingraph{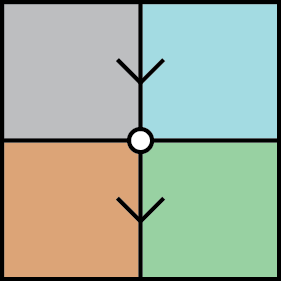}$ is called \emph{cartesian} if for any $\ingraph{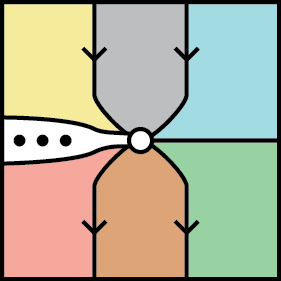}$, there exists a unique $\ingraph{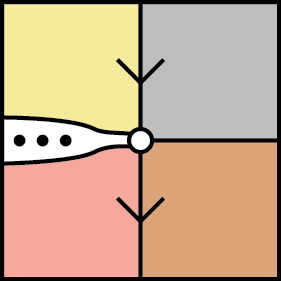}$ so that
        $$\fullgraph{2CellvComp} = \fullgraph{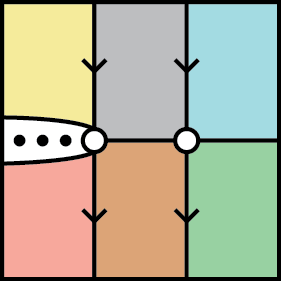}.$$
    \end{definition}

The proarrow $\ingraph{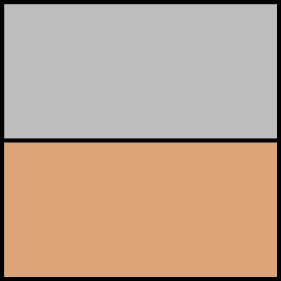}$ is determined uniquely up to isomorphism in a cartesian cell $\ingraph{2QWCG}$. For this reason, we give it the canonical name $K(g, f)$ and call it the \emph{restriction} of $K = \ingraph{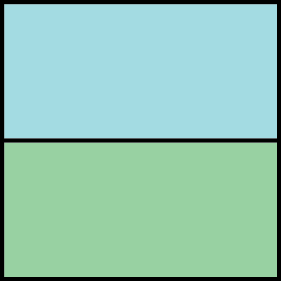}$ along $g = \tinygraph{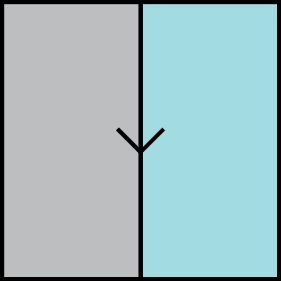}$ and $f = \tinygraph{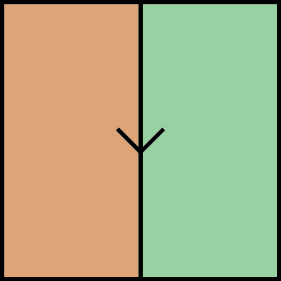}$. Note that the restriction of $\ingraph{hCG}$ by identities is itself.

\begin{definition}
A \emph{virtual equipment} is a virtual double category in which every object has a unit and the restriction $K(g, f)$ exists for every compatible triple of $K = \tinygraph{hCG}$, $g = \tinygraph{vQC}$, and $f = \tinygraph{vWG}$.
\end{definition}

We can use the universal property of restrictions to find the conjoint and companion bends for an arrow $\tinygraph{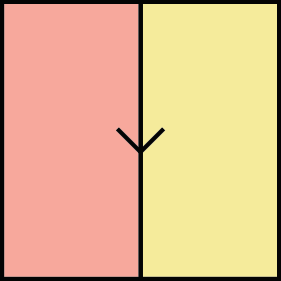}$. We'll define the conjoint here; the companion works similarly.

Given an arrow $\tinygraph{vRY}$ in a virtual equipment, define its conjoint $\tinygraph{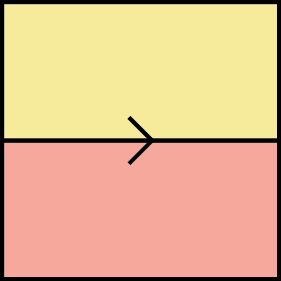}$ as the restriction of $\tinygraph{oY}$ along $\tinygraph{oY}$ and $\tinygraph{vRY}$.\footnote{Remember, we refer to the object $\tinygraph{oY}$ and its unit by the same diagram; here we are referring to the unit.} We write the defining cartesian cell as $\tinygraph{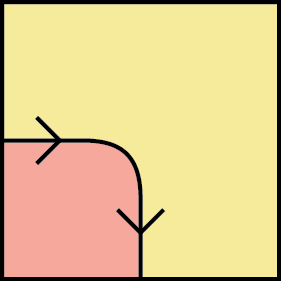}$. By the universal property of the restriction, the 2-cell $\tinygraph{vRY}$ factors through $\tinygraph{Bend_BL_RY}$ uniquely. Write this unique factor as $\tinygraph{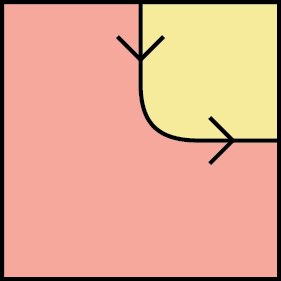}$, so that the factorization reads
$$\ingraph{vRY} = \ingraph{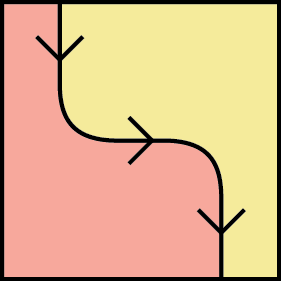}.$$
This is one of the kink identities for the conjoint. The other kink lemma from the universal property of $\tinygraph{Bend_BL_RY}$ as well. The cell $\ingraph{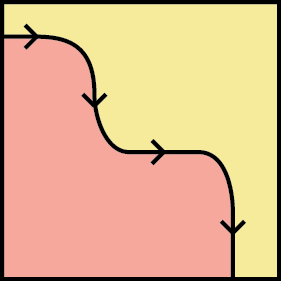}$ factors through $\tinygraph{Bend_BL_RY}$ in two ways:
$$\ingraph{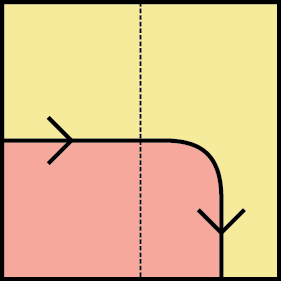} = \ingraph{Virtual_BendFactor2} = \ingraph{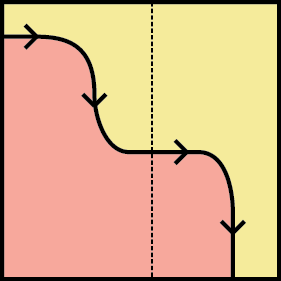}.$$
Thefore, by the uniqueness part of the universal property of $\tinygraph{Bend_BL_RY}$, we  must have
$$\ingraph{hYRr} = \ingraph{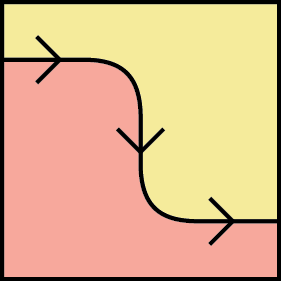}.$$
These equalities justify the visual representation of the companion and conjoint as bends from vertical to horizontal. Using the kink identities for the companion and conjoint, we show can show the restriction $K(g, f)$ of $K$ by $g$ and $f$ is the composite of of $K$ with the companion of $g$ and conjoint of $f$, justifying our use of the diagram $\ingraph{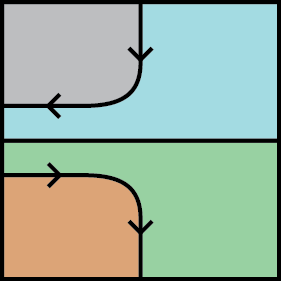}$ for the defining cartesian cell of any restriction.
\begin{lemma}
For a proarrow $\tinygraph{hCG}$ and arrows $\tinygraph{vQC}$ and $\tinygraph{vWG}$ which admit the restriction $\tinygraph{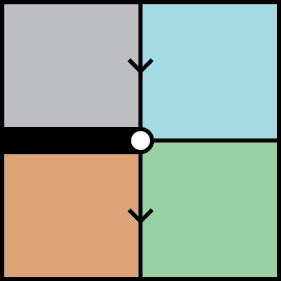}$ in a virtual equipment, $\tinygraph{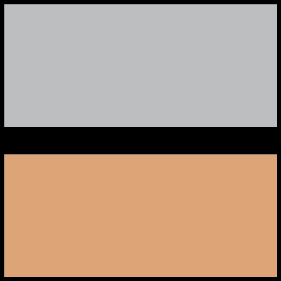}$ is the composite of $\tinygraph{hCG}$ with $\tinygraph{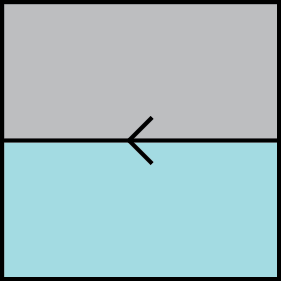}$ and $\tinygraph{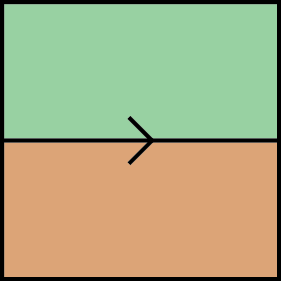}$.
\end{lemma}
\begin{proof}
The cell $\ingraph{VirtualRestriction}$ factors uniquely through $\tinygraph{Virtual_RestrictionCell}$ as $\ingraph{VirtualRestriction} = \ingraph{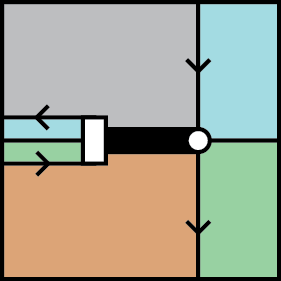}$, so we will take the structure morphism of the composite to be $\tinygraph{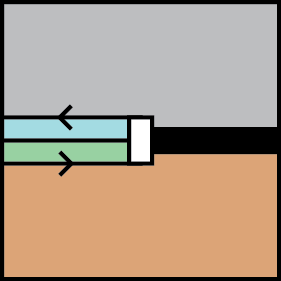}$. It remains to show that it satisfies the universal property. 
Given a 2-cell $\ingraph{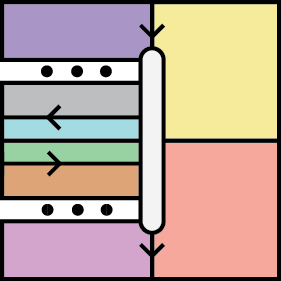}$, we can factor it through $\tinygraph{Virtual_Restriction_Defn2}$ as
$$\fullgraph{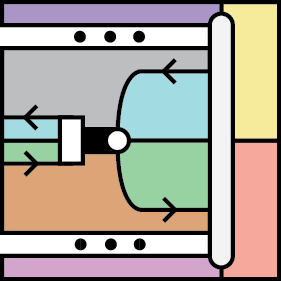} = \fullgraph{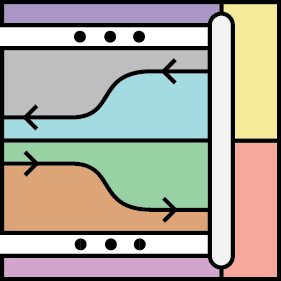} = \fullgraph{Virtual_RestrictionComp_Cell}.$$
The uniqueness of this factorization follows from the uniqueness of $\tinygraph{Virtual_Restriction_Defn2}$.
\end{proof}
\begin{corollary}\label{Lem:CompCompConj}
Given two arrows $\tinygraph{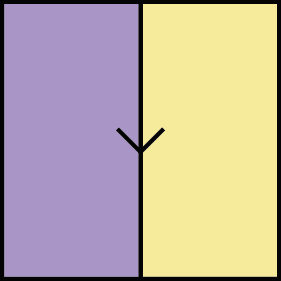}$ and $\tinygraph{vRY}$ in a virtual equipment, the composite $\tinygraph{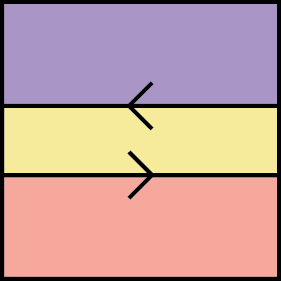}$ of their respective companion and conjoint exists.
\end{corollary}

A similar argument can be used to establish the following lemma.
\begin{lemma}\label{Lem:CompCompComp}
Given $\tinygraph{vBY}$ and $\tinygraph{vYR}$ in a virtual equipment, the composite of their companions exists and equals the companion of their composite. Similarly, the composite of their conjoints exists and equals the conjoint of their composite.
\end{lemma}

By putting together Lemma \ref{Lem:VirtualComp}, Corollary \ref{Lem:CompCompConj}, and Lemma \ref{Lem:CompCompComp}, we can show that the composite of a proarrow with any number of companions on top and any number of conjoints on bottom exists.
\begin{lemma}
In a virtual equipment, the composite of a proarrow with any number of companions on top and any number of conjoints on bottom exists and is the restriction of the proarrow by the respective vertical composites.
\end{lemma}

\section{Categories Enriched in a Virtual Equipment}\label{sec:enriched.cats}

In this section, we construct the virtual equipment of categories enriched in a virtual equipment. We begin by recalling the requisite definitions, which are due to Leinster \cite{Leinster}.

A category $\Ca$ enriched in a virtual equipment $\Ea$ consists of the following data:
\begin{itemize}
    \item A class of objects $\Ca_0$, with each object $A \in \Ca_0$ associated with an object $\Ca(A) = \tinygraph{oY}$ in $\Ea$ called its \emph{extent}. We will refer to an object of $A$ by the same color as we refer to it's extent; nevertheless, two objects of $\Ca$ may have the same extent. We may therefore be referring to the same object in $\Ea$ using two different colors.
    \item For each pair of objects $\tinygraph{oY}$ and $\tinygraph{oR}$ in $\Ca_0$, a proarrow $\Ca(\tinygraph{oY}, \tinygraph{oR}) = \tinygraph{hYR}$ in $\Ea$.\
    \item For each object $\tinygraph{oY}$ in $\Ca_0$, a $2$-cell $\id_{\nanograph{oY}} = \tinygraph{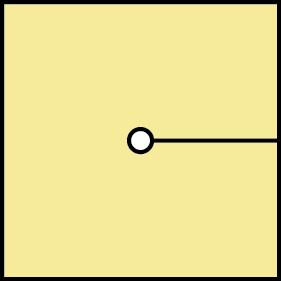}$ called the \emph{identity}.
    \item For each triple of objects $\tinygraph{oY}$, $\tinygraph{oR}$, $\tinygraph{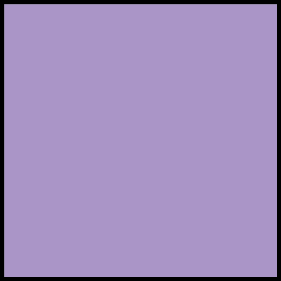}$, a 2-cell $\tinygraph{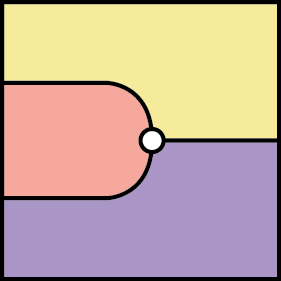}$ called \emph{composition}.
\end{itemize}
This data satisfies the identity and associativity laws:
$$\fullgraph{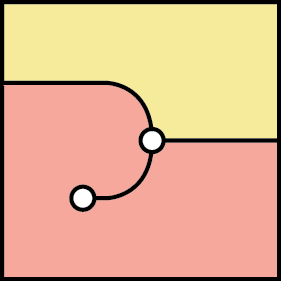} = \fullgraph{hYR} = \fullgraph{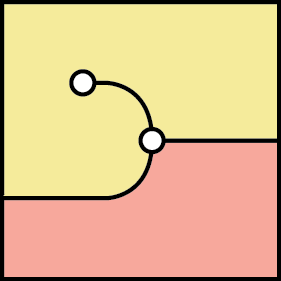},$$
$$\fullgraph{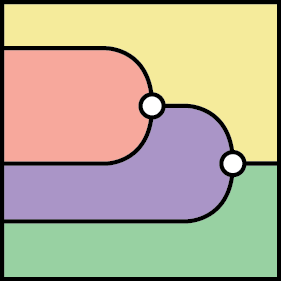} = \fullgraph{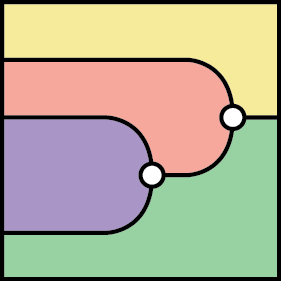}.$$

A functor $F : \Ca \to \Da$ between enriched categories consists of the following data:
\begin{itemize}
    \item For each object $\tinygraph{oY}$ in $\Ca$, an object $\tinygraph{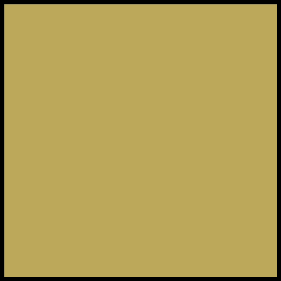}$ in $\Da$ and an arrow $\tinygraph{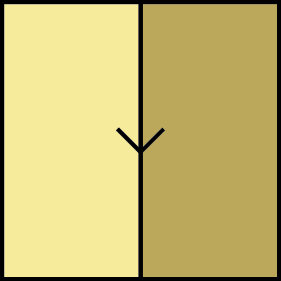}$ in $\Ea$ from the extent of $\tinygraph{oY}$ to the extent of $\tinygraph{Virtual_darkY}$.
    \item For each pair of objects $\tinygraph{oY}$ and $\tinygraph{oR}$ of $\Ca$, a 2-cell $\tinygraph{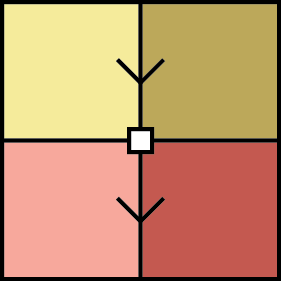}$ in $\Ea$.
\end{itemize}
This data satisfies the functor laws:
$$\fullgraph{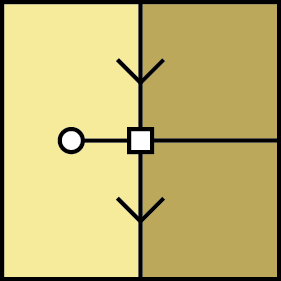} = \fullgraph{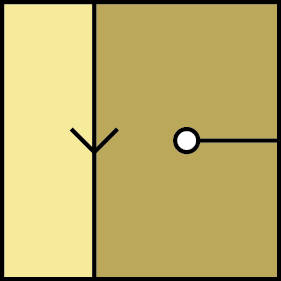}, \quad\mbox{and}\quad\fullgraph{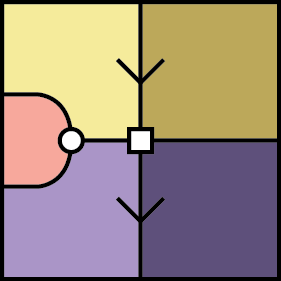} = \fullgraph{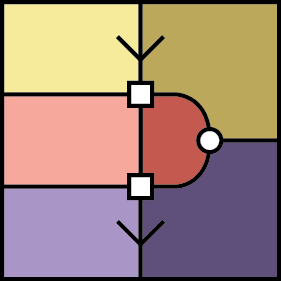}.$$
These laws can be summarized by imagining that the vertical functorial strings lie over the horizontal categorical ones, so that the law above expresses a braiding of a sort. 

A profunctor $J : \Ca \topro \Da$ consists of the following data:
\begin{itemize}
    \item For each $\tinygraph{oY} \in \Ca$ and $\tinygraph{oR} \in \Da$, a proarrow $\tinygraph{hYR}$ in $\Ea$.
    \item For each pair $\tinygraph{oY}$ and $\tinygraph{oB}$ in $\Ca$, a 2-cell $\tinygraph{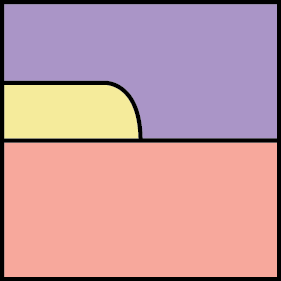}$ in $\Ea$. Similarly, for each pair $\tinygraph{oR}$ and $\tinygraph{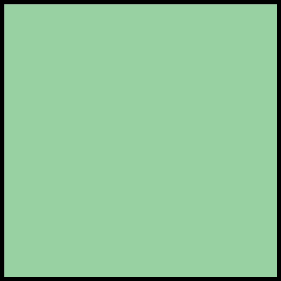}$ in $\Da$, a 2-cell $\tinygraph{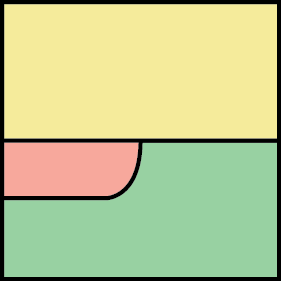}$ in $\Ea$.
\end{itemize}
This data satisfies the profunctor laws:
$$\fullgraph{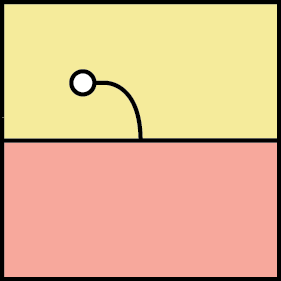} = \fullgraph{hYR}, \quad\quad\quad\quad\fullgraph{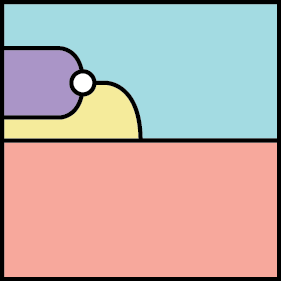} = \fullgraph{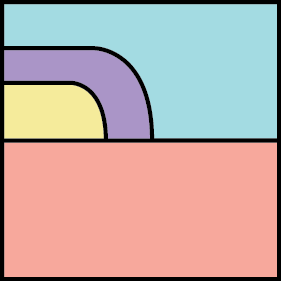}, \quad\mbox{and for $\tinygraph{Virtual_Prof2},$}$$
$$\fullgraph{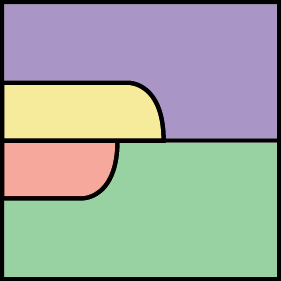} = \fullgraph{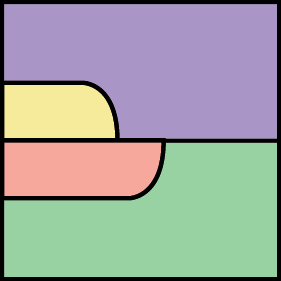}.$$

Finally, in order to show that enriched categories may be arranged into a virtual equipment, we need to describe a 2-cell. Let $\Ca_0$, $\ldots$, $\Ca_k$, $\Da_0$ and $\Da_1$ be enriched categories, let $F : \Ca_0 \to \Da_0$ and $G : \Ca_k \to \Da_1$ be functors, and let $J_i : \Ca_{i-1} \topro \Ca_i$ and $K : \Da_0 \topro \Da_1$ be profunctors. A morphism $\alpha$ with this signature is a collection of  2-cells $\alpha_{\vec{c}} : J_1(c_0, c_1)\cdots J_k(c_{k-1}, c_k) \xrightarrow[G]{F} K(Fc_0, Gc_k)$ in $\Ea$, which may be written as $\ingraph{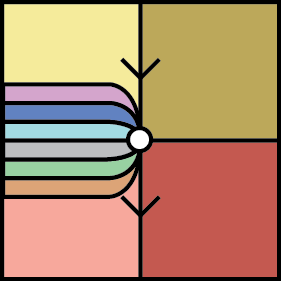}$. This family satisfies two families of laws:
$$(1)\quad \fullgraph{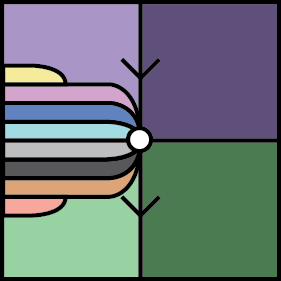} = \fullgraph{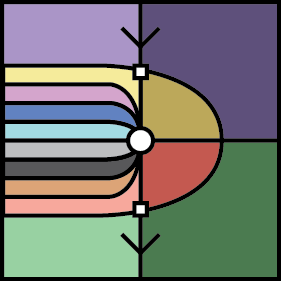},\quad\mbox{and} \quad (2)\quad \fullgraph{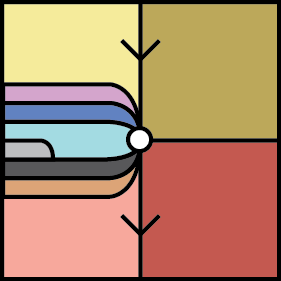} = \fullgraph{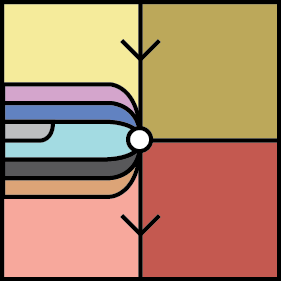}.$$
These laws say that the action of the categories $\Ca_i$ on the profunctors $J_i$ are equalized by $\alpha$ (if $1 < i < k$), or that the action commutes with the action of the functors $F$ and $G$. The identity morphism of a profunctor has identities for all of its components.

Composition of profunctors is given by composing their components in $\Ea$. The associativity and identity laws of a virtual equipment then follow from the same laws in $\Ea$.

Likewise, the restrictions in $\VCat{\Ea}$ are given by taking the restrictions component-wise. In other words, we define the component of $K(G, F)$ at $\tinygraph{oY}$ and $\tinygraph{oR}$ as $\ingraph{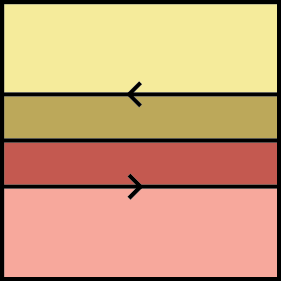}$. We define the action with the cell $\ingraph{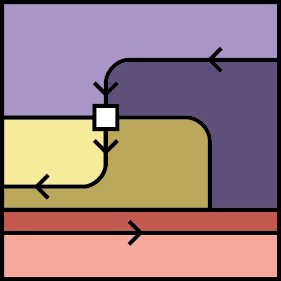}$ using the components of the functor $F$ and the action on $K$. The other action is defined similarly; we will only work with the top action. That this is indeed a profunctor follows from the functor laws of $F$ and profunctor laws of $K$:
$$\fullgraph{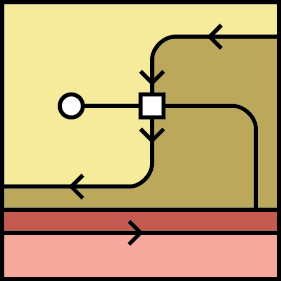} = \fullgraph{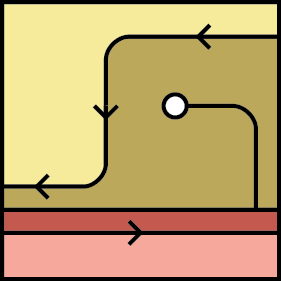} = \fullgraph{Virtual_ProfRestr},$$
$$\fullgraph{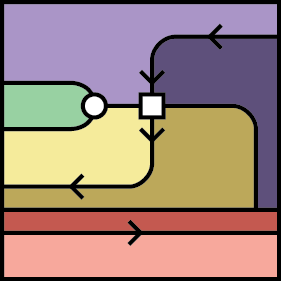} = \fullgraph{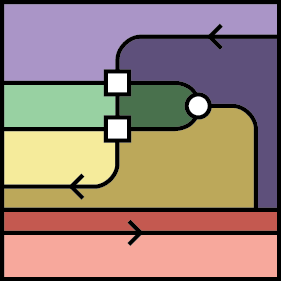} = \fullgraph{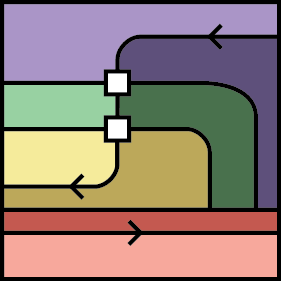} = \fullgraph{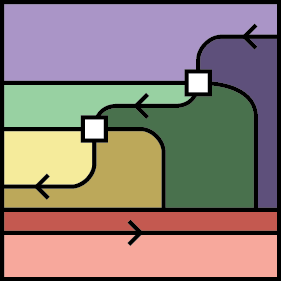}.$$

The cartesian cell of $K(G, F)$ is given component-wise by $\ingraph{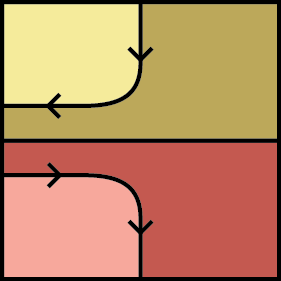}$. That this is a profunctor morphism follows quickly from the definition of the action and the kink identities. $K(G, F)$ as defined satisfies the required universal property because its components do.

\section{The Canonical Embedding}\label{sec:embedding}
 Let's begin to describe the functor $| \cdot | : \Ea \to \VCat{\Ea}$ embedding a virtual equipment into the virtual equipment of categories enriched in it, which I'll call the `canonical embedding'. The construction has the flair of a Yoneda embedding, which I hope to explore further in subsequent work. Though we've taken the codomain of this functor to be the virtual equipment of categories enriched in $\Ea$ and all functors and profunctors between them, the realization functor will land in the part of $\VCat{\Ea}$ that uses only purely horizontal cells of $\Ea$ -- functors won't change the extent of the objects they act on.

 For an object $\ingraph{oY}$ of $\Ea$, we define its representation (its image under $| \cdot |$) to be the $\Ea$-enriched category with
\begin{itemize}
    \item Objects vertical arrows $\ingraph{vRY}$, with each object's extent being its domain.
    \item Between objects $\ingraph{vBY}$ and $\ingraph{vRY}$, a hom-object $\ingraph{Embed_Hom_BYR}$ (in $\Ea$).
    \item For object $\ingraph{vRY}$, an identity arrow \ingraph{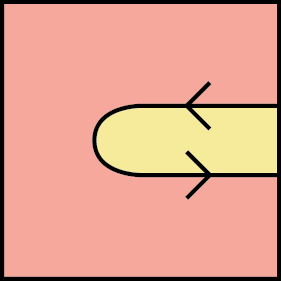} (a 2-cell in $\Ea$).
    \item For each composable triple, a composition arrow \ingraph{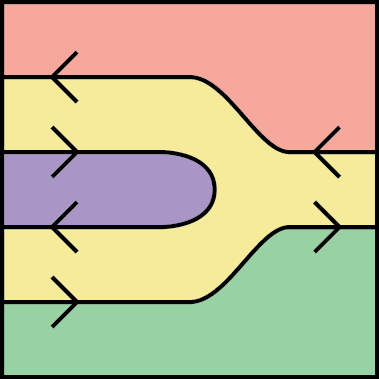} (a 2-cell in $\Ea$).
\end{itemize}

We can verify that the identity and associativity conditions hold graphically.
$$\fullgraph{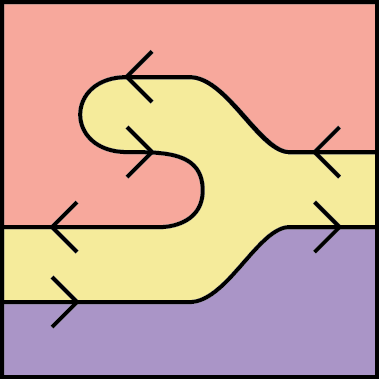} = \fullgraph{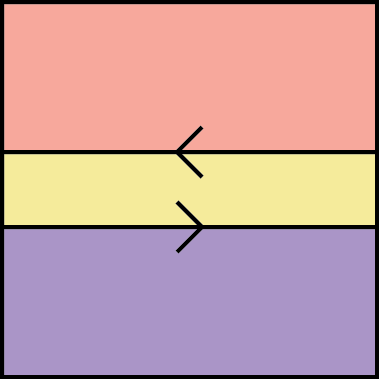} = \fullgraph{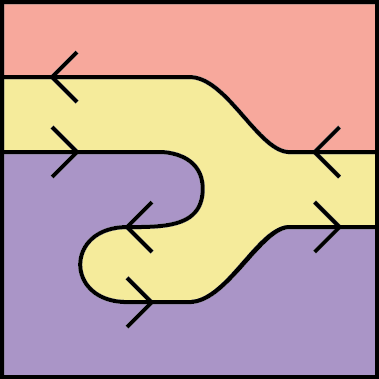}.$$
$$\fullgraph{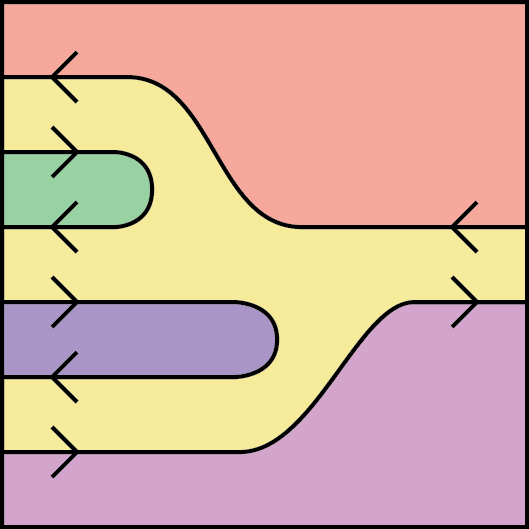} = \fullgraph{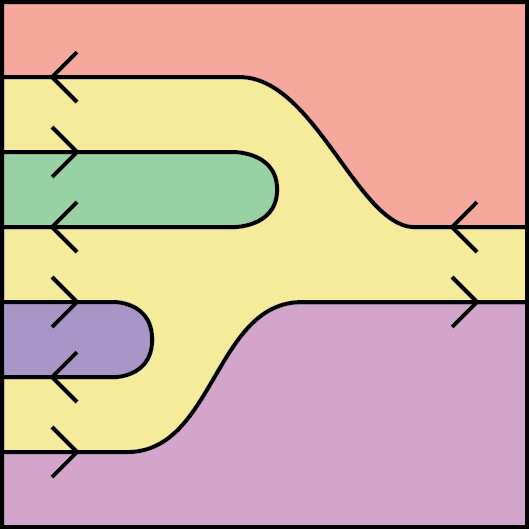}.$$
So this construction does indeed yield a category $\realize{\nanograph{oY}}$ enriched in $\Ea$. It remains to realize vertical arrows as functors, proarrows as profunctors, and 2-cells as morphisms of profunctors. We begin by recalling the definition of a functor between enriched categories in the case we are concerned with. The functor we define here have only trivial action on the extents of objects.

\begin{definition}\label{Def_Functor_Embed}
A functor $f : \realize{\nanograph{oY}} \to \realize{\nanograph{oR}}$ which does not change extent consists in
\begin{enumerate}
    \item For each element $x = \ingraph{vBY}$ of $\realize{\nanograph{oY}}$, an element $f(x) = \ingraph{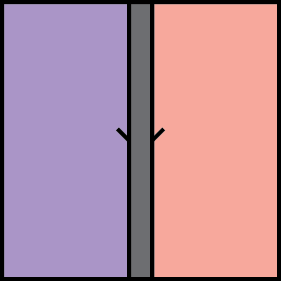}$ of $\realize{\nanograph{oR}}$ with the same domain.
    \item For each pair of elements $\ingraph{vBY}$ and $\ingraph{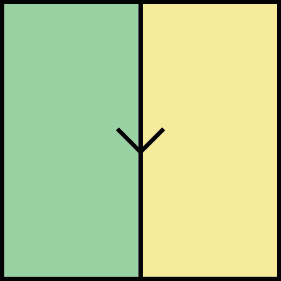}$, a 2-cell $\ingraph{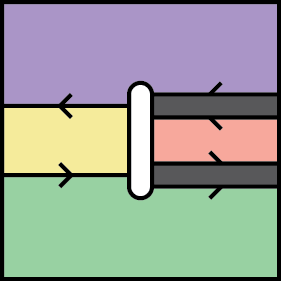}$, such that the following functor laws hold:
    $$\fullgraph{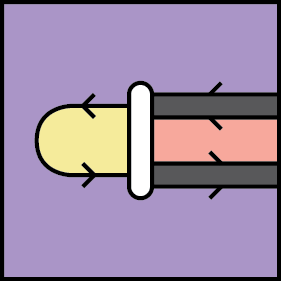} = \fullgraph{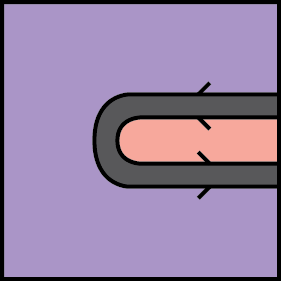}$$
    $$\fullgraph{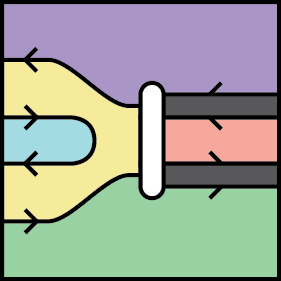} = \fullgraph{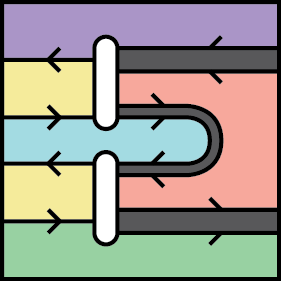}$$
\end{enumerate}
\end{definition}

For an arrow $\ingraph{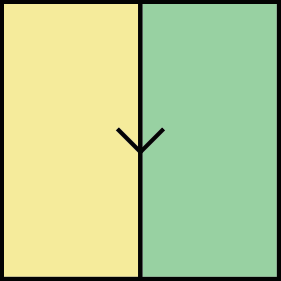}$, we get an enriched functor sending objects $\ingraph{vRY}$ to $\ingraph{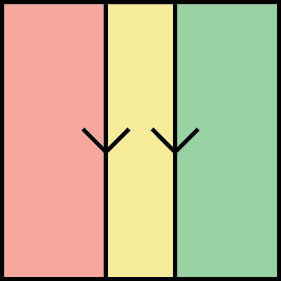}$ and acting on homs by $\ingraph{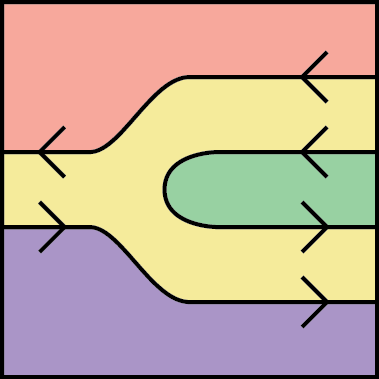}$. We can check that this satisfies the functor conditions graphically.

$$\fullgraph{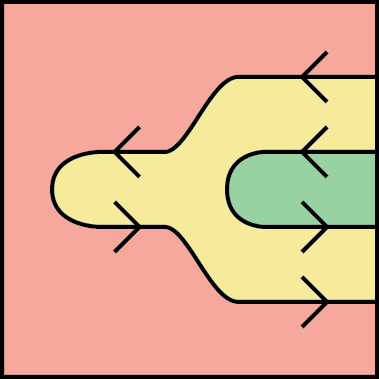} = \fullgraph{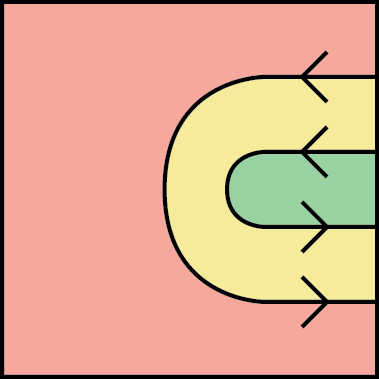}$$
$$\fullgraph{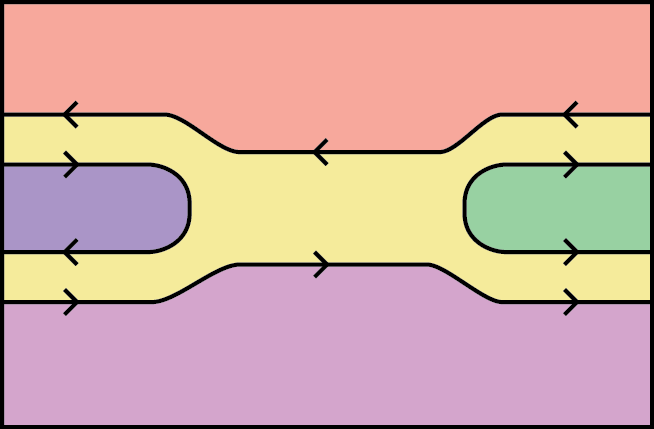} = \fullgraph{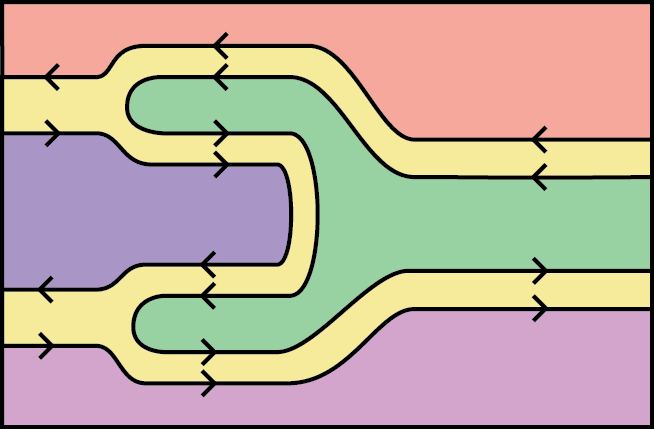}$$

Though it is a special case of the following discussion of the representation of a general 2-cell, we will take a moment to discuss the representation of a vertical 2-cell as a natural transformation. For a 2-cell $\ingraph{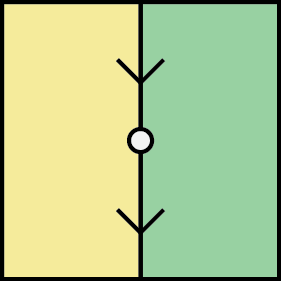}$, and for each object $\ingraph{vRY}$ of $\ingraph{oY}$, we get a component $\ingraph{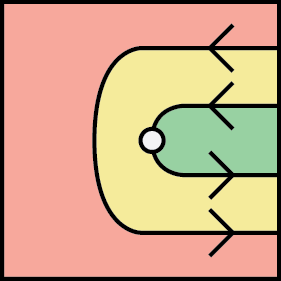}$ of a natural transformation. This satisfies the laws for a natural transformation, as we can see graphically.

$$\fullgraph{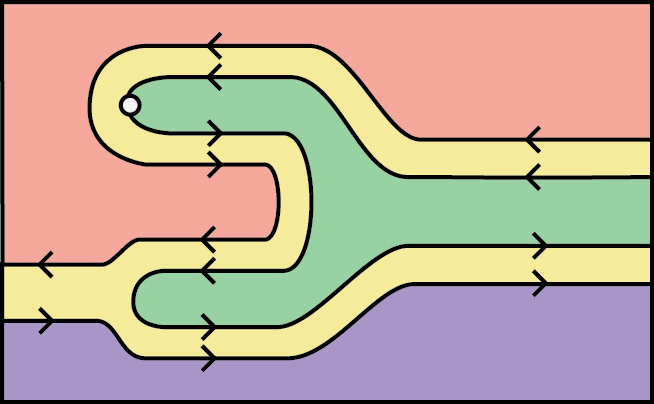} = \fullgraph{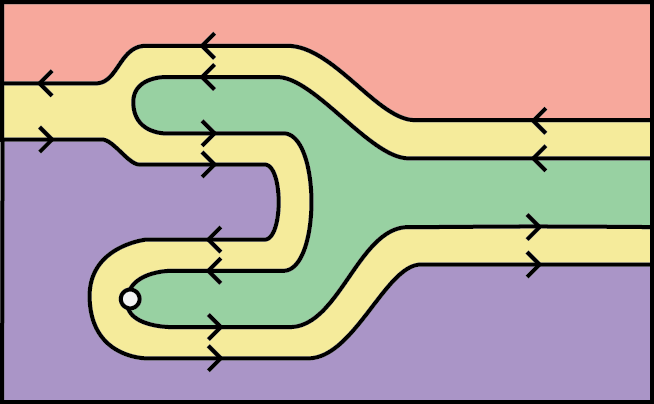}$$

For purposes of explicit calculation, its useful to note that the above two morphisms are equal to $\ingraph{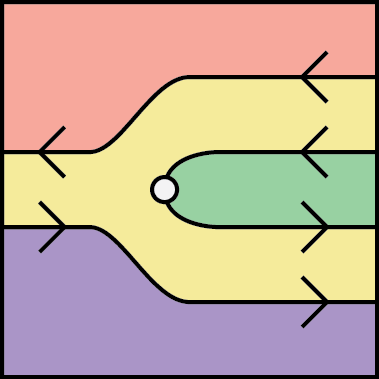}$. This completes the construction of the functor $| \cdot | : v\Ea \rightarrow \VCat{\Ea}$ from the vertical 2-category $v\Ea$ of $\Ea$. Now we can turn our attention to proarrows.

For each proarrow $\ingraph{hYR}$ in $\Ea$, we get the enriched profunctor sending $\tinygraph{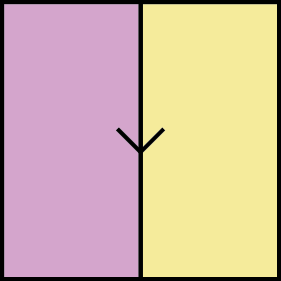}$ and $\tinygraph{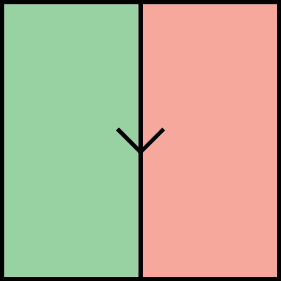}$ to $\ingraph{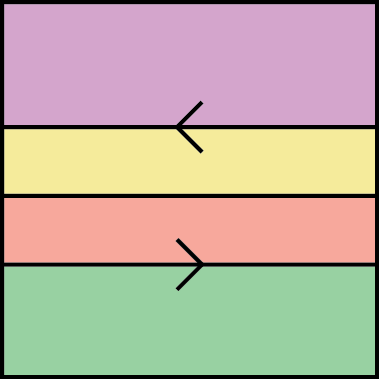}$. The category induced by $\tinygraph{oY}$ acts on the representative of $\tinygraph{hYR}$ via the map $\ingraph{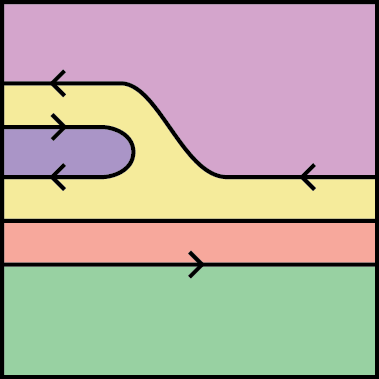}$, subject to the following rules that may be verified graphically:
\begin{align*}
    \fullgraph{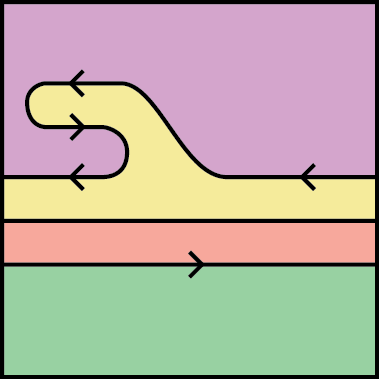} &= \fullgraph{Embed_Prof_BYRG} \\
    \fullgraph{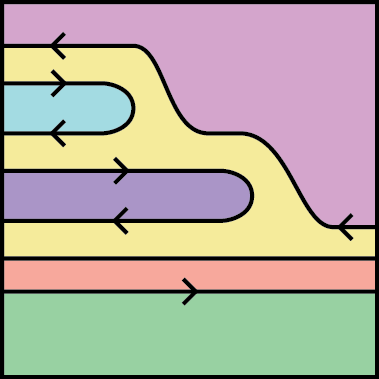} &= \fullgraph{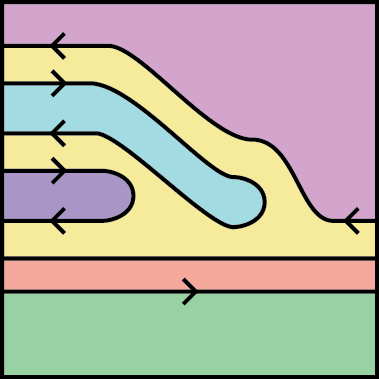}.
\end{align*}

The representative of $\tinygraph{oR}$ acts similarly by the map $\ingraph{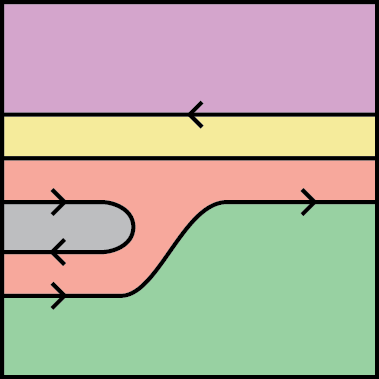}$, subject to analogous laws as those above. The two actions interact according to the following rule, which holds graphically:
$$\fullgraph{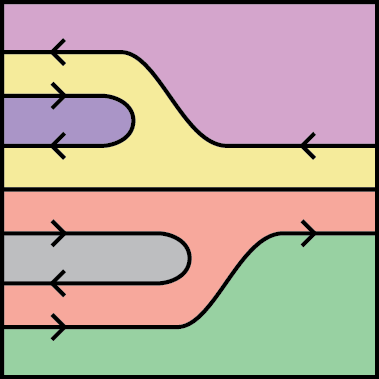} = \fullgraph{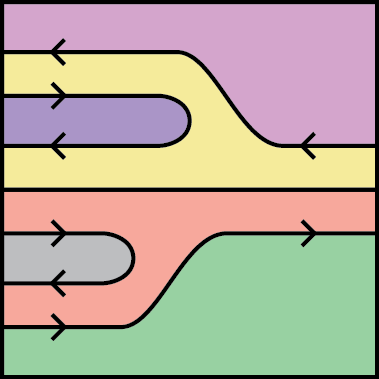}.$$

Finally, we come to the representation of a 2-cell. We begin by recalling the definition of a morphism of enriched profunctors in the relevant case.

\begin{definition}\label{Def_ProMorph_Embed}
A morphism of profunctors $\alpha : |\tinygraph{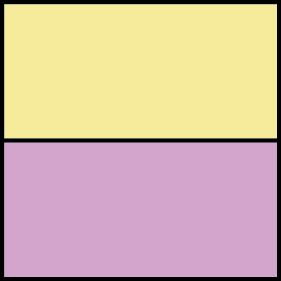}| \cdots |\tinygraph{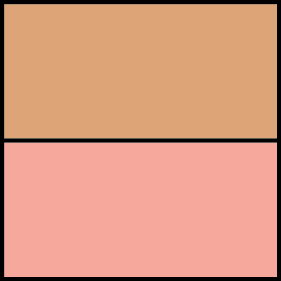}| \xrightarrow[|\tinygraph{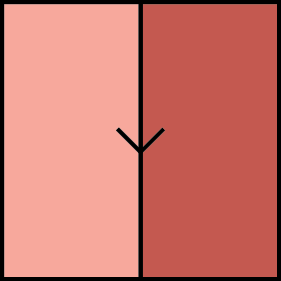}|]{|\tinygraph{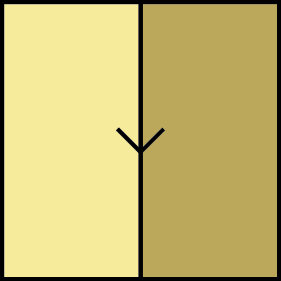}|} |\tinygraph{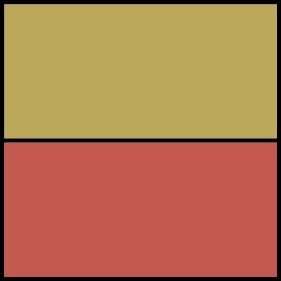}|$ is a family of $2$-cells $\alpha_{\vec{a}}$ for each list $\vec{a} = \tinygraph{vBY},\,\tinygraph{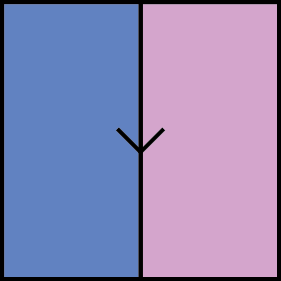},\ldots,\tinygraph{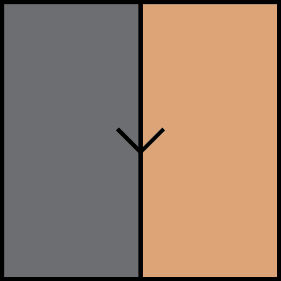},\,\tinygraph{vGR}$ in $\Ea$ of the following form:
$$\fullgraph{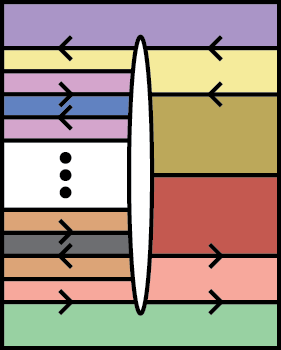}.$$

This family follows the following two equations
$$(1)\quad \fullgraph{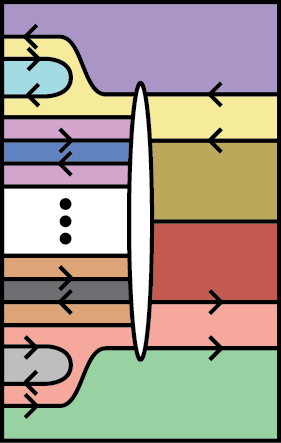} = \fullgraph{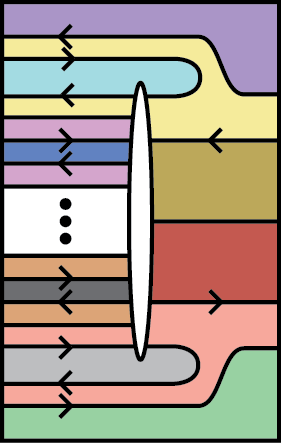} \quad \text{and} \quad (2) \quad \fullgraph{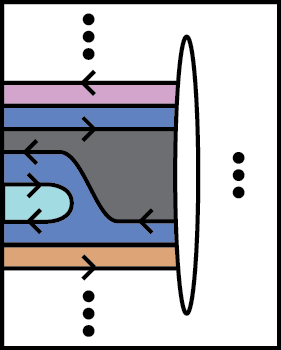} = \fullgraph{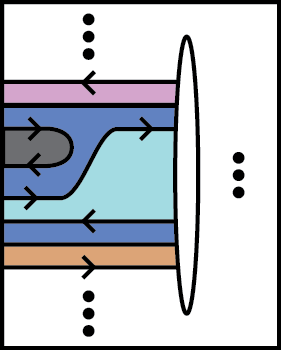}$$

\end{definition}

The representative of a two cell $\ingraph{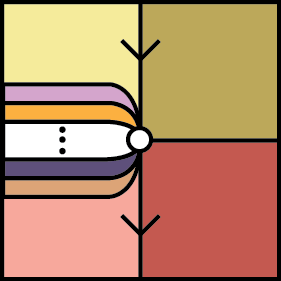}$ is a map of profunctors which assigns to each list $\tinygraph{vBY},\,\tinygraph{vB2P},\ldots,\tinygraph{vQdW},\,\tinygraph{vGR}$ of objects the 2-cell
$$\fullgraph{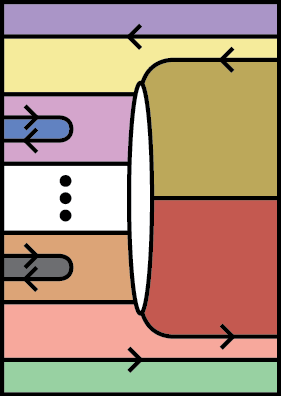}.$$
This can be quickly seen to satisfy the necessary equations
$$(1)\quad \fullgraph{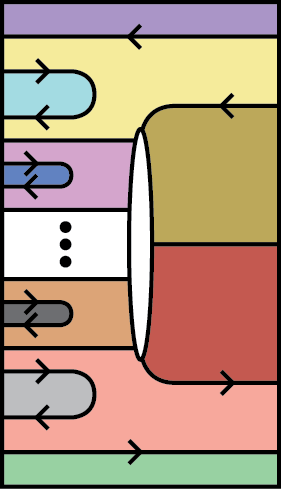} = \fullgraph{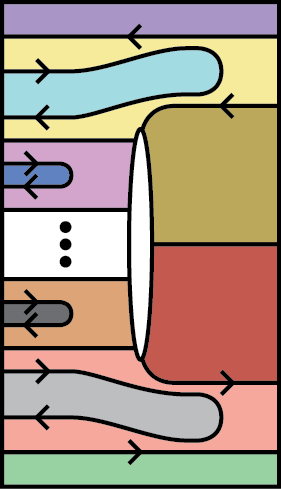} \quad \text{and} \quad (2) \quad \fullgraph{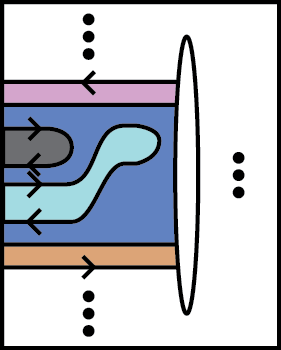} = \fullgraph{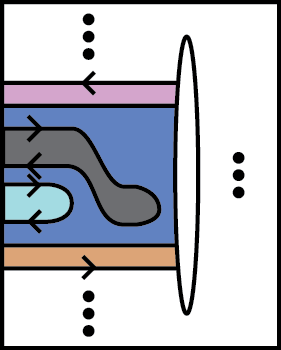}.$$

From the pictures, it is clear that the representative of the various identity 2-cells are their respective identities: simply delete the lines you don't need! 

\begin{remark}
If our virtual equipment is a virtual equipment of categories enriched in a monoidal equipment, and we restrict the domains of the objects of $|\nanograph{oY}|$ to the walking object (one object category whose hom is the monoidal unit), then the above construction will recover $\tinygraph{oY}$. In other words, if our objects are already enriched categories, then we can recover them from their representatives.
\end{remark}

\section{Proving Functoriality}\label{sec:functoriality}

In this section, we prove the functoriality of the canonical embedding constructed above, and show that it preserves composites of proarrows. 

We will just prove functoriality for 2-cells with two input proarrows, since these contain all elements that appear in 2-cells with arbitrarily many inputs. The proof for general 2-cells is the same, but with added ellipses.

Consider the composite $\ingraph{Virtual_CompCellRepComposed}$ of the 2-cell $\ingraph{Virtual_CompCellRep1}$ with $\ingraph{Virtual_CompCellRep2}$ and $\ingraph{Virtual_CompCellRep3}$. At $\tinygraph{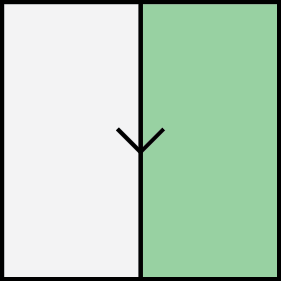}$, $\tinygraph{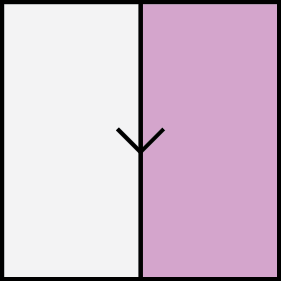}$, $\tinygraph{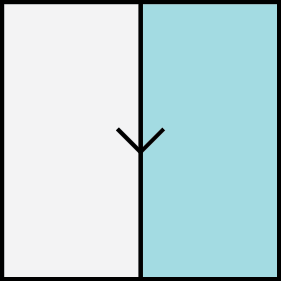}$, $\tinygraph{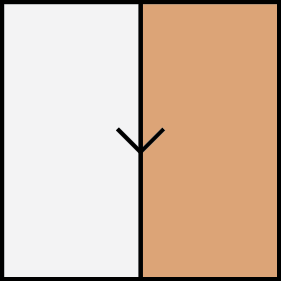}$ and $\tinygraph{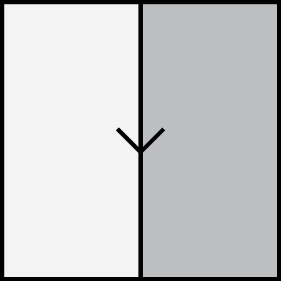}$ (domains left unspecified), the component of the composite of the representatives is as on the left, and the component of the representative of the composite is as on the right:
$$\fullgraph{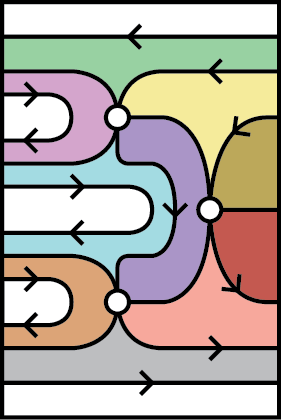} =  \fullgraph{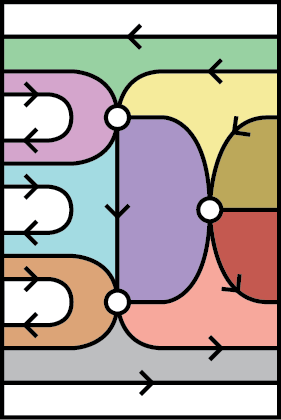}.$$
Since composites of profunctor morphisms are taken componentwise, this shows that the composite of the representatives is the representative of the composite of 2-cells.

Now we turn to composites of proarrows. These are not necessarily preserved by functors between virtual equipments because they are determined by a universal property. We will show that the representative of the witness of a composite in $\Ea$ is itself a witness to a composite in $\VCat{\Ea}$.

\begin{proposition}
The representative (image under the canonical embedding $|\cdot| : \Ea \to \VCat{\Ea}$) of a witness to a composite in $\Ea$ is itself the witness of a composite in $\VCat{\Ea}$.
\end{proposition}
\begin{proof}
Again, we suppose that the composite is formed of two proarrows; the general case is analogous. Suppose that $\ingraph{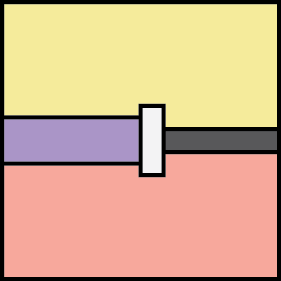}$ is the witness of a composite proarrow. Given any morphism of profunctors with the representatives of $\tinygraph{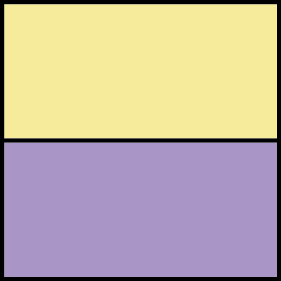}$ and $\tinygraph{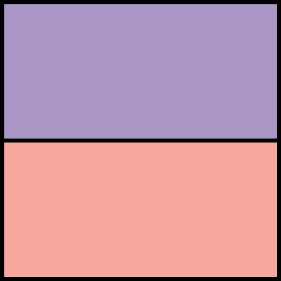}$ as inputs (possibly among others), we may factor it component-wise in $\Ea$ through the composite witness as follows\footnote{We draw a single proarrow flanking the two representatives on each side. The general case is analogous.}:
$$\fullgraph{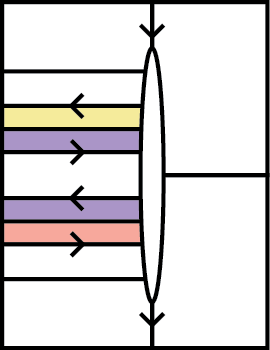} =\fullgraph{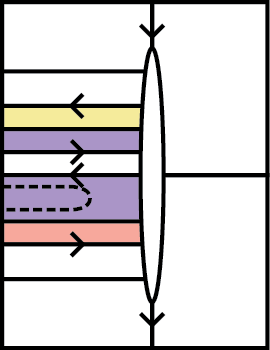}=
\fullgraph{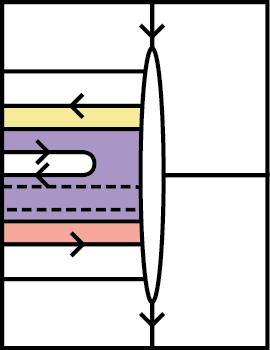} = \fullgraph{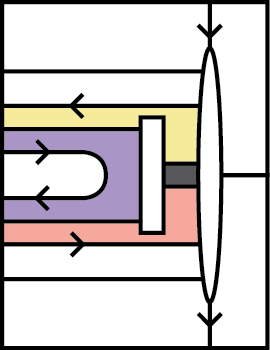}$$

From the diagram on the right, it can be seen that we are factoring through the representative of $\ingraph{Embed_Prof_CompWitness}$. 

It remains to show that the other 2-cell in the diagram (guaranteed by the universal property of $\tinygraph{Embed_Prof_CompWitness}$) forms a component of a profunctor morphism. We do this by factoring both sides of the defining equalities of the original profunctor morphism through $\tinygraph{Embed_Prof_CompWitness}$, and then using the uniqueness of its universal property.
\begin{enumerate}
    \item $$\fullgraph{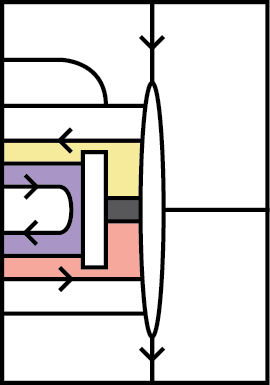} =
    \fullgraph{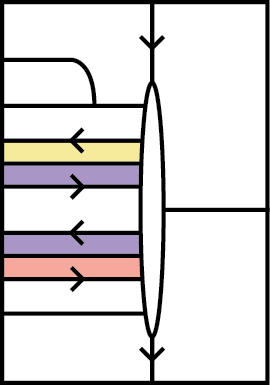} =
    \fullgraph{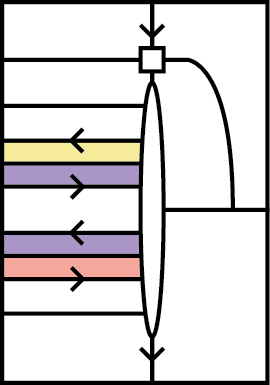} =
    \fullgraph{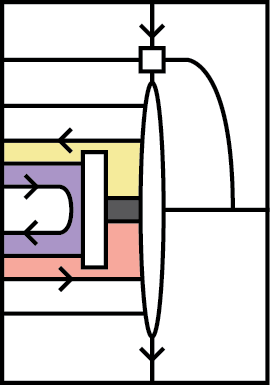}$$
    \item $$\fullgraph{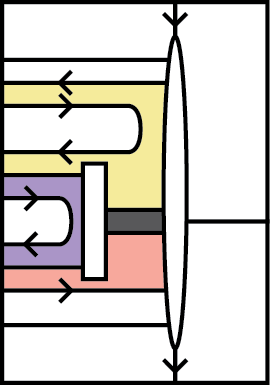} =
    \fullgraph{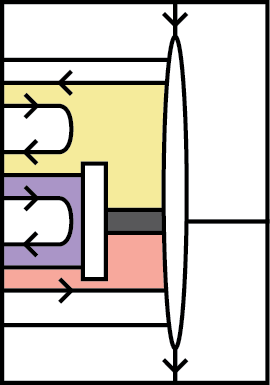} =
    \fullgraph{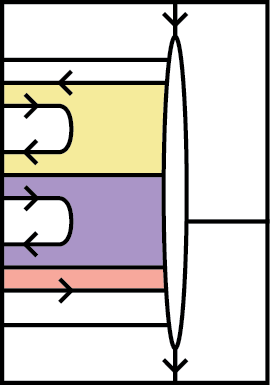} =
    \fullgraph{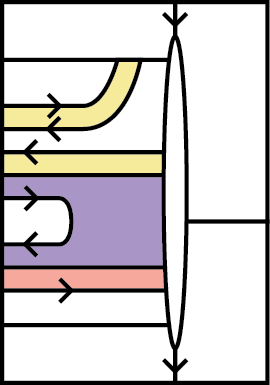} =
    \fullgraph{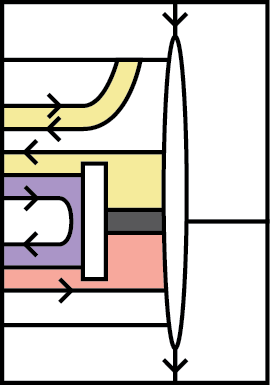} $$
\end{enumerate}
\end{proof}

\section{The Image of the Canonical Embedding}\label{sec:properties}
In this section, we briefly investigate the image of the canonical embedding. We will see that $|\cdot|$ is full and faithful 2-cells, full on arrows, and ``nearly full'' on proarrows in the sense that the action on proarrows is coreflective. This justifies are calling $|\cdot|$ an `embedding'. All the proofs involve taking the components at the element of $|A|$ represented by the identity at $A$, which gives the embedding the smell of a Yoneda embedding. This also means that they would need to be significantly adapted to deal with restrictions of the domains of elements of the representative of $A$, say in the case that we restrict to elements whose domain is a monoidal identity for $\Ea$.

We can easily show that the canonical embedding is full and faithful on $2$-cells. This in particular means that it is faithful, up to isomorphism, on vertical and horizontal arrows. For a refresher on profunctor morphisms between representatives, see Definition \ref{Def_ProMorph_Embed}.
\begin{proposition}
The canonical embedding $| \cdot | : \Ea \to \VCat{\Ea}$ is full and faithful on $2$-cells.
\end{proposition}
\begin{proof}
We will work with $2$-cells with two proarrows in the domain, since the general case is the same. 

Let's begin with fullness. Suppose we have a profunctor morphism $\alpha : |\nanograph{hYB}|,\,|\nanograph{hBR}| \xrightarrow[|\nanograph{vRRd}]{|\nanograph{vYYd}} |\nanograph{hYdRd}|$. Taking the component of $\alpha$ at the identities of $\tinygraph{oY}$, $\tinygraph{oB}$, and $\tinygraph{oR}$ respectively, we get a 2-cell in $\Ea$ which we may bend into $\ingraph{Virtual_CompCellRep1}$. It remains to show that $\alpha = |\tinygraph{Virtual_CompCellRep1}|$, which we can by arguing component-wise:
$$\fullgraph{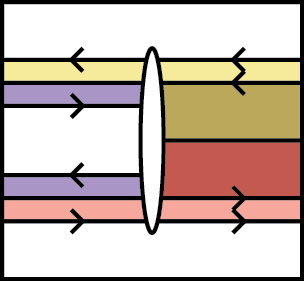} = \fullgraph{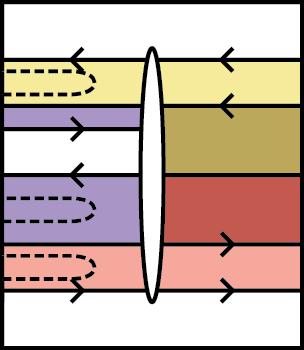} = \fullgraph{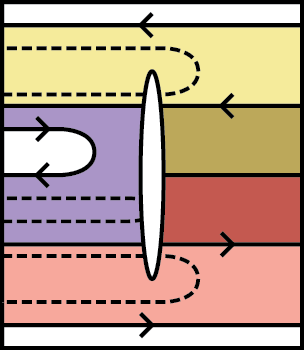} = \fullgraph{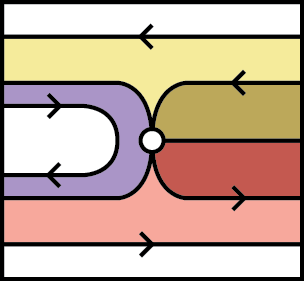}.$$

Faithfulness follows quickly. If two 2-cells have the same representation, then those representations are equal component-wise and so in particular at the identity components. But these 2-cells at the identity components of the representations are the 2-cells being represented, so the two 2-cells were equal to begin with.
\end{proof}

The canonical embedding is also full on arrows. 
\begin{proposition}
The canonical embedding $| \cdot | : \Ea \to \VCat{\Ea}$ is full on arrows.
\end{proposition}
\begin{proof}
Combine the upcoming lemmas \ref{Lem_EmbeddingArrowEqualsFlat} and \ref{Lem_EmbeddingFullNoExtent}.
\end{proof}

We will show this in two parts, but begin first with a useful definition. 
\begin{definition}
Let $f : \Ca \to |\nanograph{oR}|$ be an enriched functor sending $x \in \Ca$ with $\Ca(x) = \tinygraph{oY}$ to $f(x) = \ingraph{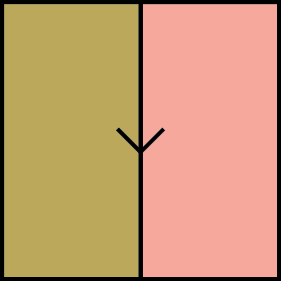}$ with structure map $\tinygraph{vYYd}$ and acting on homs by $\ingraph{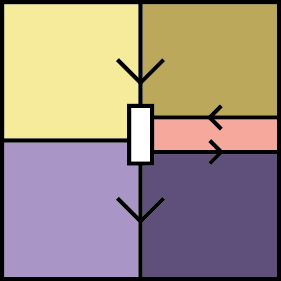}$. Define $f^{\flat} : \Ca \to |\tinygraph{oR}|$ to be the enriched functor given by sending $x$ to $\ingraph{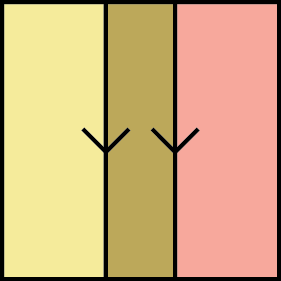}$ and acting on homs by $\ingraph{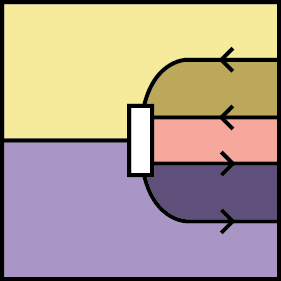}$. Note that the structure maps of $f^{\flat}$ on objects are identities; we say that it \emph{does not change extent}.
\end{definition}

Now we show that every functor into a representative is isomorphic to one that does not change extent.
\begin{lemma}\label{Lem_EmbeddingArrowEqualsFlat}
Let $f, f^{\flat} : \Ca \to |\tinygraph{oR}|$ be as above. Then the maps $\ingraph{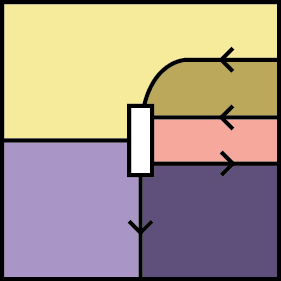}$ and $\ingraph{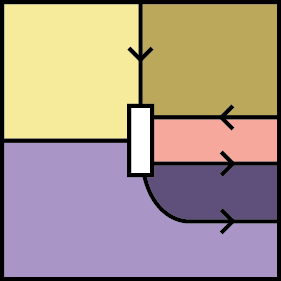}$ form the components of a natural isomorphism $f \cong f^{\flat}$.
\end{lemma}
\begin{proof}
Follows quickly from the functor laws.
\end{proof}

For a refresher on the definition of a functor between representatives, see Definition \ref{Def_Functor_Embed}.

\begin{lemma}\label{Lem_EmbeddingFullNoExtent}
The canonical embedding is full on functors which do not change the extent of objects (that is, whose structure maps on objects are identities).
\end{lemma}
\begin{proof}
Let $f : \realize{\nanograph{oY}} \to \realize{\nanograph{oR}}$ be such a functor, and let $\ingraph{vYR}$ be $f(\id_{\nanograph{oY}})$. We will construct a natural isomorphism $|f(\id_{\nanograph{oY}})| \cong f$. This involves first constructing, for any $x = \ingraph{vBY}$, an isomorphism $f(x) = \ingraph{Embed_Full_Arrow_fBR} \cong \ingraph{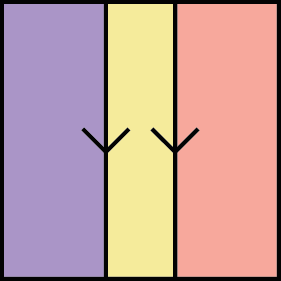}$, and then proving naturality. First, we note that $\ingraph{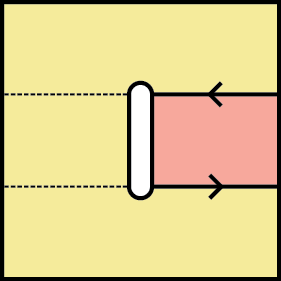} = \ingraph{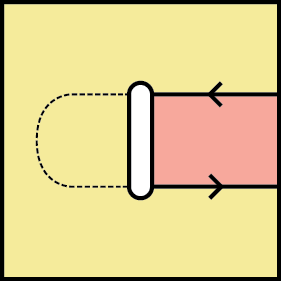} = \ingraph{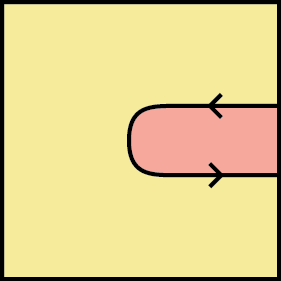}$ by the identity functor law.

The data of $f$ gives us two 2-cells $\ingraph{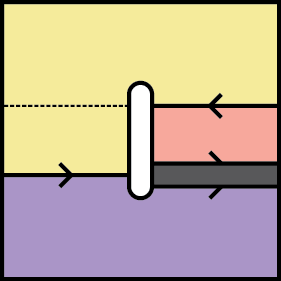}$ and $\ingraph{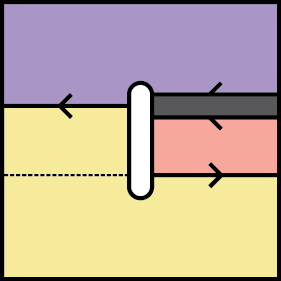}$. By bending, we get 2-cells $\ingraph{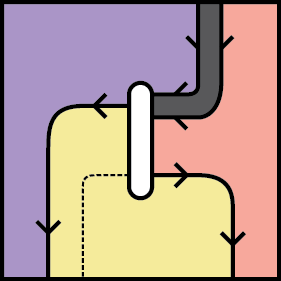}$ and $\ingraph{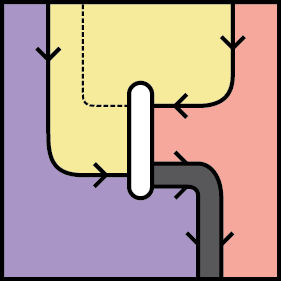}$, which we can show are inverse using the functor laws as follows.
$$\fullgraph{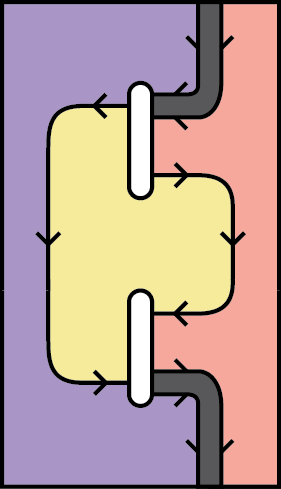} = \fullgraph{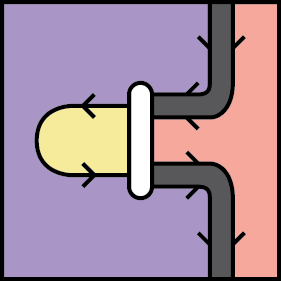} = \fullgraph{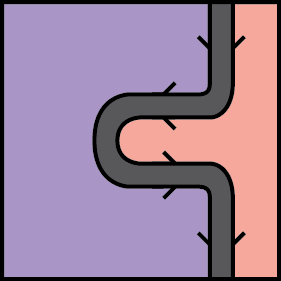} = \fullgraph{Embed_Full_Arrow_fBR}$$
$$\fullgraph{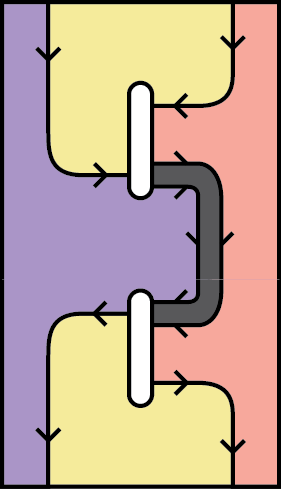} = \fullgraph{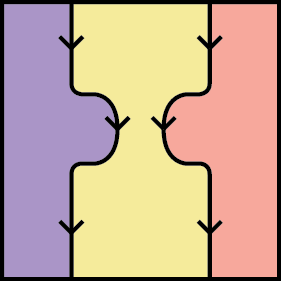} = \fullgraph{vCompBYR}$$

It remains to show that this isomorphism is natural. We will show that $\ingraph{Embed_Full_xid}$ is natural through a quick calculation; the naturality of $\ingraph{Embed_Full_idx}$ follows similarly.
$$\fullgraph{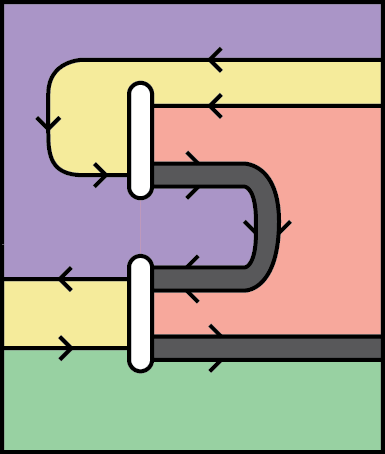} = \fullgraph{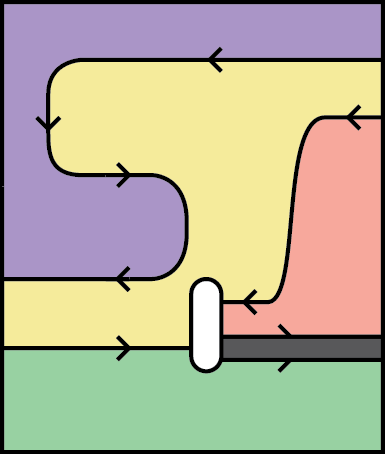} = \fullgraph{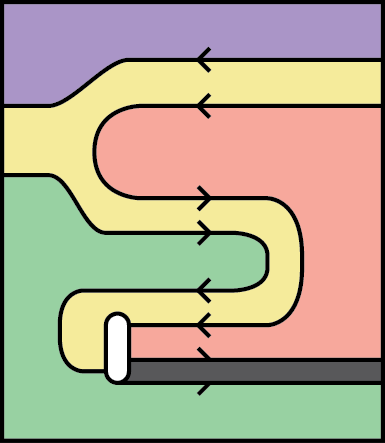}$$

\end{proof}

We now characterize the realization of proarrows. While $|\cdot|$ is not full on proarrows, it is nearly full in the following precise way.

\begin{proposition}\label{Prop_ProFull}
The realization functor exhibits $h\Ea(A, B)$ as a coreflective subcategory of $h(\VCat{\Ea})(|A|, |B|)$, with coreflector given by $J \mapsto \overline{J} := J(\id_A, \id_B)$.
\end{proposition}
\begin{proof}
Given a profunctor $J: |\nanograph{oY}| \topro |\nanograph{oR}|$, we can take its component at the identities of $\tinygraph{oY}$ and $\tinygraph{oR}$ to get a proarrow $\ingraph{hYR}$ in $\Ea$. That the actions $\ingraph{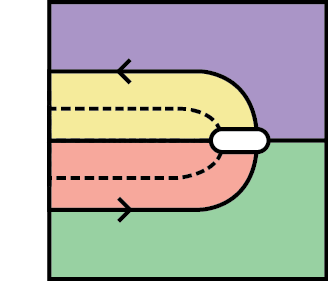}$ form the components of a morphism $|J(\nanograph{oY}, \nanograph{oR})| \to J$ in $\VCat{\Ea}$ follows immediately from the profunctor laws governing $J$. Naturality in $J$ immediately follows from the commutation of the action with the components of a morphism $J \to J'$.

Note that $|K|(\id_{\nanograph{oY}}, \id_{\nanograph{oR}}) = K$ and that the counit $||K|(\id_{\nanograph{oY}}, \id_{\nanograph{oR}})| \to |K|$ is therefore an identity (or unitor). Likewise, the unit $J(\id_{\nanograph{oY}}, \id_{\nanograph{oR}}) \to |J(\id_{\nanograph{oY}}, \id_{\nanograph{oR}})|(\id_{\nanograph{oY}}, \id_{\nanograph{oR}})$ is an identity (or unitor). Therefore, the zig-zag identities hold trivially.
\end{proof}

\begin{corollary}
The essential image of $|\cdot|$ on profunctors between $|A|$ and $|B|$ consists of precisely those profunctors $J : |A| \topro |B|$ for which the action $|J(\id_A, \id_B)| \to J$ is a natural isomorphism.
\end{corollary}
\begin{proof}
This is a fact about coreflective subcategories in general. See, e.g., Proposition 1.3 of \cite{Gabriel}. 
\end{proof}

\begin{proposition}\label{Prop_ProFullFunc}
If the composite $JK$ of $J : |A| \topro |B|$ and $K : |B| \topro |C|$ exists, then $\overline{JK} \cong \overline{J}\,\overline{K}$.
\end{proposition}
\begin{proof}
The isomorphism is given by the structure map $J(\id_A, \id_B) K(\id_B, \id_C) \to JK(\id_A, \id_C)$ and the map $JK(\id_A, \id_C) \to  J(\id_A, \id_B) K(\id_B, \id_C)$ guaranteed by applying the universal property to the identity of $|J(\id_A, \id_B) K(\id_B, \id_A)|$ taken at the $\id_A$, $\id_C$ component. The universal property of $JK$ guarantees that these are mutually inverse.
\end{proof}

While the canonical embedding is not full on proarrows, it does reflect Morita equivalences.

\begin{corollary}
If $|A|$ and $|B|$ are Morita equivalent (that is, admitting an equivalence in the horizontal bicategory of composable profunctors), then $A$ and $B$ are horizontally equivalent.
\end{corollary}
\begin{proof}
Assume, for the sake of this proposition, that the requisite composites of profunctors exists in $\VCat{\Ea}$ so that $|A|$ and $|B|$ may be Morita equivalent. The claim then follows from Propositions \ref{Prop_ProFull} and \ref{Prop_ProFullFunc}.

That $|\cdot|$ is conservative on profunctors between $|A|$ and $|B|$ follows from a fact about coreflective subcategories in general (see, e.g., Proposition 1.3 of \cite{Gabriel}). Combined with the functoriality proved in Proposition \ref{Prop_ProFullFunc}, we see that if $JK \cong \id_{|A|}$ and $KJ \cong \id_{|B|}$, then $\overline{J}\,\overline{K} \cong \id_A$ and $\overline{K}\,\overline{J} \cong \id_B$.
\end{proof}